\documentclass[smallextended,referee,envcountsect,]{svjour3}
\smartqed
\usepackage{graphicx}
\journalname{JOTA}

\usepackage[breaklinks=true,colorlinks=true,linkcolor=red, citecolor=blue, urlcolor=blue]{hyperref}%

\usepackage{amsfonts}
\usepackage{amssymb}
\usepackage{dsfont}

\usepackage{xcolor}
\usepackage{colortbl}
\usepackage{color}
\usepackage{graphicx}
\usepackage{multirow}
\usepackage{pifont}
\usepackage[most]{tcolorbox}

\usepackage{algorithm}
\usepackage{algorithmic}

\allowdisplaybreaks

\usepackage{empheq}
\usepackage[font=small]{caption}
\usepackage[font=small]{subcaption}

\usepackage[algo2e]{algorithm2e} 
\usepackage{verbatim}
\usepackage{xspace} %

\usepackage{enumerate}
\usepackage{enumitem}

\newcommand{\dotprod}[2]{\left\langle #1,#2 \right\rangle}
\newcommand{\norms}[1]{\left\| #1 \right\|}

\newcommand{\expect}[1]{\mathbb{E}\left[ #1 \right]}

\definecolor{PineGreen}{HTML}{008B72}

  \providecommand{\bb}{\mathbf{b}}

  \providecommand{\ee}{\mathbf{e}}

  \renewcommand{\gg}{\mathbf{g}}

  \providecommand{\tt}{\mathbf{t}}
  \providecommand{\uu}{\mathbf{u}}

  \usepackage{bm}

\usepackage[textsize=tiny]{todonotes}

\providecommand{\mycomment}[3]{\todo[caption={},size=footnotesize,color=#1!20, inline]{\textbf{#2: }#3}}%
\providecommand{\inlinecomment}[3]{%
  {\color{#1}#2: #3}}%
\newcommand\commenter[2]%
{%
  \expandafter\newcommand\csname i#1\endcsname[1]{\inlinecomment{#2}{#1}{##1}}
  \expandafter\newcommand\csname #1\endcsname[1]{\mycomment{#2}{#1}{##1}}
}

\newcommand{\xk}{x_{k}}

\newcommand{\xkk}{x_{k+1}}

\newcommand{\yk}{y_{k}}

\newcommand{\vkk}{z_{k+1}}

\newcommand{\agk}{\alpha_{k}}
\newcommand{\bgk}{\zeta_{k}}
\newcommand{\ak}{a_{k}}
\newcommand{\bk}{b_{k}}

\newcommand{\bkk}{b_{k+1}}
\newcommand{\gk}{\gamma_{k}}

\newcommand{\circledOne}{\text{\ding{172}}}
\newcommand{\circledTwo}{\text{\ding{173}}}
\newcommand{\circledThree}{\text{\ding{174}}}
\newcommand{\circledFour}{\text{\ding{175}}}  
\newcommand{\circledFive}{\text{\ding{176}}}

\usepackage{url}

\commenter{AL}{blue}
\newcommand{\al}[1]{{\color{black}#1}}

\begin{document}

\title{The ``Black-Box'' Optimization Problem: Zero-Order Accelerated Stochastic Method via Kernel~Approximation }

\titlerunning{Zero-Order Accelerated Stochastic Method via Kernel Approximation}


\author{Aleksandr Lobanov \and Nail Bashirov \and Alexander Gasnikov}

\institute{\textbf{Aleksandr Lobanov},  Corresponding author \at
              Moscow Institute of Physics and Technology \\
              \textit{9 Institutskiy per., Dolgoprudny, 141701, Russian Federation} \\
              Skolkovo Institute of Science and Technology \\
              \textit{30 Bolshoy Boulevard, bld. 1, Moscow, 121205, Russian Federation} \\
              ISP RAS Research Center for Trusted Artificial Intelligence\\
              \textit{25 A. Solzhenitsyn st., Moscow, 125047, Russian Federation}\\
              lobbsasha@mail.ru
           \and
              \textbf{Nail Bashirov}  \at
              Moscow Institute of Physics and Technology \\
              \textit{9 Institutskiy per., Dolgoprudny, 141701, Russian Federation} \\
              Institute for Information Transmission Problems\\
              \textit{19 B. Karetny per., Moscow, 127051, Russian Federation}\\
              bashirov.nr@phystech.edu
            \and
              \textbf{Alexander Gasnikov} \at 
              Moscow Institute of Physics and Technology \\
              \textit{9 Institutskiy per., Dolgoprudny, 141701, Russian Federation} \\
              Innopolis University \\
              \textit{1 Universitetskaya Str., Innopolis, 420500, Russian Federation} \\
              ISP RAS Research Center for Trusted Artificial Intelligence\\
              \textit{25 A. Solzhenitsyn st., Moscow, 125047, Russian Federation}\\
              gasnikov@yandex.ru
}


\date{Received: date / Accepted: date}

\maketitle

\begin{abstract}
In this paper, we study the standard formulation of an optimization problem when the computation of gradient is not available. Such a problem can be classified as a ``black box'' optimization problem, since the oracle returns only the value of the objective function at the requested point, possibly with some stochastic noise. Assuming convex, and higher-order of smoothness of the objective function, this paper provides a zero-order accelerated stochastic gradient descent (ZO-AccSGD) method for solving this problem, which exploits the higher-order of smoothness information via kernel approximation. As theoretical results, we show that the ZO-AccSGD algorithm proposed in this paper improves the convergence results of state-of-the-art (SOTA) algorithms, namely the estimate of iteration complexity. In addition, our theoretical analysis provides an estimate of the maximum allowable noise level at which the desired accuracy can be achieved. Validation of our theoretical results is demonstrated both on the model function and on functions of interest in the field of machine learning. We also provide a discussion in which we explain the results obtained and the superiority of the proposed algorithm over SOTA algorithms for solving the original problem.
\end{abstract}

\section{Introduction}

Black-box optimization problems \cite{Kimiaei_2022} (or also known as derivative-free optimization problems \cite{Conn_2009,Rios_2013}) arise when the gradient computation process is unavailable for some reason, e.g., the objective function $f(x)$ is not smooth \cite{Polyak_1969,Dvinskikh_2022,Huang_2023,Lobanov_MOTOR} or the process of computing the gradient $\nabla f(x)$ is too ``expensive'' compared to computing the value of the objective function $f(x)$ (in this case, the functions can be either smooth \cite{Ajalloeian_2020,Akhavan_2022} or of higher order of smoothness \cite{Bach_2016,Akhavan_2020,Lobanov_2023}). Moreover, there are often situations in practice \cite{Bogolubsky_2016} when the oracle returns a noisy value of the objective function (i.e., the value of the function $f(x)$ with some bounded noise $\tilde{\xi}$) at the requested point $x$, where the noise directly affects the ``cost'' of calling the oracle: the more inaccurate the oracle returns the value of the objective function (i.e., the greater the noise), the cheaper the oracle call. Such an oracle has a common ``charactonym'' name in the literature, namely a gradient-free oracle or a zero-order oracle \cite{Rosenbrock_1960}.  Since this class of optimization problem has significant interest in settings such as federated learning \cite{Lobanov_2022,Patel_2022}, distributed learning \cite{Akhavan_2021,Yu_2021,Lobanov_NET}, overparameterized models \cite{Lobanov_overparametrization} (in particular, in application problems such as hyperparameter tuning \cite{Li_2017,Hazan_2017}, multi-armed bandits \cite{Flaxman_2004,Bartlett_2008}, and many others \cite{Nguyen_2023}...), it is important to know and understand what approaches exist to solve this~class~of~problem.

Apparently, the main way to solve the black-box problem is to apply gradient-free algorithms/zero-order methods to this problem. Among such methods there are two classes (or two approaches to the creation of gradient-free algorithms): the first is the class of evolutionary algorithms \cite{Storn_1997,Auger_2005,Hansen_2006}, which can often show their efficiency only empirically; the second is the class based on the advantages of first-order algorithms \cite{Kiefer_1952,Gasnikov_ICML}, whose efficiency is provided in the form of theoretical estimates. Evolutionary algorithms are often used in the class of non-convex multimodal problems, in which the main goal is to find not a local but a global optimum. However, in the class~of~convex~problems, it makes sense to use theoretically based algorithms, which~often~guarantee faster convergence by taking advantage of first-order~algorithms.

The basic idea of creating efficient gradient-free algorithms for solving the convex black-box optimization problem is to use instead of the true gradient in first-order optimization algorithms some estimate of the gradient or also known as a gradient approximation \cite{Gasnikov_2022}. It is this seemingly simple idea that allows gradient-free algorithms to utilize the power of efficient first-order methods to solve the black box problem. However, it is important to correctly choose the algorithm and gradient approximation based on the original problem. Often, accelerated batched methods for solving corresponding optimization problems are chosen as efficient first-order algorithms, but there are also exceptions, e.g., for the class of problems satisfying the Polyak--Lojasiewicz condition, unaccelerated algorithms are already considered efficient (see \cite{Yue_2022} for more details). Regarding the question of the choice of gradient approximation: in \cite{Scheinberg_2022} it is shown that central finite difference is a more preferable scheme for constructing a gradient approximation than forward finite difference. Also in \cite{Lobanov_2023} in the Experiments section, the authors have shown on a model practical experiment the advantage of using randomized approximations, in particular among the randomized approximations, they highlight $l_2$ randomization. However, this approximation is the most preferable for solving a smooth black-box optimization problem, but not for a problem with higher-order of smoothness. In 1990, B. Polyak and A. Tsybakov managed to propose such a gradient approximation, which takes into account the advantage of higher-order of smoothness of the function \cite{Polyak_1990}. This gradient approximation is called the Kernel approximation. What distinguishes this approximation from $l_2$ randomization is the presence of a kernel by which the information about the higher-order of smoothness of the function is taken into account. And it was this paper \cite{Polyak_1990} that became the starting point for the study of the solution of the black box problem with the assumption that the function has a high order of smoothness.  

This paper investigates the improvement of the iteration complexity of gradient-free algorithms for solving a class of convex black-box optimization problems, assuming that the objective function has higher-order of smoothness. \al{To create an optimal zero-order optimization method in terms of iteration complexity, we use the accelerated batched stochastic gradient descent method \cite{Vaswani_2019} (Nesterov--accelerated) as a basis}. Among the gradient approximations that take into account the advantage of higher-order of smoothness, we choose the Kernel approximation because it is the one that requires only two calls to the gradient-free oracle per iteration (which guarantees a better estimate of oracle complexity), unlike the higher-order finite-difference gradient approximation \cite{Berahas_2022}. However, the Kernel approximation is not a biased gradient estimator, so it is important to understand how noise is accounted for in the first-order algorithm. To this end, we generalize the accelerated first-order algorithm \cite{Vaswani_2019} to the case with a biased gradient oracle. Thus, to create a gradient-free algorithm, we base on the accelerated first-order method with a biased gradient oracle (see Section \ref{sec:First_order}), using the Kernel approximation with a one-point zero-order oracle instead of the true gradient. In addition, we explicitly derive the estimate at the maximum noise level at which the desired accuracy can be achieved. Finally, we demonstrate our theoretical results on a model example.

    \subsection{Our contribution}
    Our contribution to this paper can be summarized as follows:
    \begin{itemize}
        \item We generalize the convergence results of the accelerated batched stochastic gradient descent algorithm  (Nesterov--accelerated) \cite{Vaswani_2019} to the case with a biased gradient oracle;

        \item  We provide a novel gradient-free optimization algorithm for solving a convex black-box optimization problem under an higher-order of smoothness condition: the zero-order accelerated stochastic gradient descent method (ZO-AccSGD, see Algorithm 1). This algorithm improves existing estimates on iteration complexity by working out the batching technique: $N = \mathcal{O} \left( \varepsilon^{-1/2} \right)$. Moreover, using the well-known analysis of bias and second moment (variance) estimation, we were able to get rid of the smoothness order dependence $\beta$ of the objective function in oracle complexity: $T = \mathcal{O} \left( \frac{d^2}{\varepsilon^{2 + \frac{2}{\beta - 1}}} \right)$;

        \item We provide an analysis that includes an elaboration on the maximum allowable noise level. We show that the noise level at which the desired accuracy is still achieved depends directly on the batch size $B$;

        \item We confirm our theoretical results in Section ``Experiments'' by considering both a model problem (solving a system of linear equations), and problems of interest in machine learning.
    \end{itemize}

    \subsection{Related works}

        \paragraph{Gradient-free oracle.} The gradient approximation is typically a finite difference zero-order oracle. Therefore, every work on gradient-free oracle utilizes one or another zero-order oracle concept. For example, the works in \cite{Ajalloeian_2020} propose an oracle concept that returns the exact value of the objective function at the requested point: $\tilde{f} = f(x)$. This concept is intuitive and widely used, especially in tutorials such as introductions to optimization \cite{Polyak_1987}, etc. The following concept of gradient-free oracle, introduced in \cite{Dvinskikh_2022,Lobanov_2022}, is oriented towards practical problems and is presented as follows: the oracle returns the value of the objective function at the requested point with some bounded deterministic noise $\tilde{f} = f(x) + \delta (x)$, where $|\delta (x)| \leq \Delta$. This concept applies well to deterministic optimization problems, but we can modernize it for a stochastic optimization problem \cite{Gasnikov_ICML}: $\tilde{f} = f(x, \xi) + \delta(x)$. This stochastic variant of the concept of a gradient-free oracle with bounded deterministic noise allows us to construct a gradient approximation depending on the availability of feedback. For example, if we can call the gradient-free oracle on one realization of the function twice, then the Kernel approximation with two-point feedback takes the following form: $\gg = \frac{d}{2h} \left( f(x + h r \ee, \xi) + \delta(x + h r \ee) - f(x - h r \ee, \xi) - \delta(x + h r \ee) \right) K(r) \ee$. However, if we only have access to one-point feedback, i.e., we can call the oracle on one realization of the function only once, then the kernel approximation takes the following form: $\gg = \frac{d}{h} \left( f(x + h r \ee, \xi) + \delta(x + h r \ee) \right) K(r) \ee$. In addition, there is another concept of the gradient-free oracle that is quite controversial in gradient approximation, namely the kernel approximation \cite{Polyak_1990,Akhavan_2020,Novitskii_2021,Lobanov_2023,Akhavan_2023} is as follows: $\gg = \frac{d}{2h} \left( f(x + h r \ee) + \tilde{\xi}_1 - f(x - h r \ee) - \tilde{\xi}_2 \right) K(r) \ee$. Such a gradient approximation can quite rightly be called a one-point feedback approximation, although at first glance it is not even obvious that a gradient-free oracle that returns the value of the objective function at the requested point with some bounded stochastic noise $\tilde{f} = f(x) + \tilde{\xi}$ is a stochastic gradient-free oracle. However, if we consider the stochastic noise $\tilde{\xi}_1$ and $\tilde{\xi}_2$ as realizations of the function, it will be clear enough that such a gradient approximation can rightfully be called a one-point approximation, since the function is computed though twice in one iteration, but on different realizations. In our work, we also use the concept of a gradient-free oracle with stochastic noise, which generates a gradient approximation with one-point feedback.   

        \paragraph{Bounded gradient noises.} Currently, there are a series of works \cite{Woodworth_2021_over,Rakhlin_2012,Hazan_2014,Bertsekas_1996,Stich_2019,Lobanov_2023,Schmidt_2013,Srebro_2010} that assume different constraints on gradient noise. For example, \cite{Rakhlin_2012,Hazan_2014} uses a standard and common in earlier works constraint on gradient noise, namely, they estimate some constant: $\expect{\norms{\nabla f(x,\xi)}^2} \leq \sigma^2$. However, there is some disadvantage of such a constraint, namely the large number constraint, that is, if the norm of the gradient decreases with the number of iterations, the estimate of the second moment will remain as large. To address this problem, some works \cite{Bertsekas_1996,Stich_2019,Lobanov_2023} impose a constraint that is considered more adaptive: $\expect{\norms{\nabla f(x,\xi)}^2} \leq \rho \norms{\nabla f(x)}^2 + \sigma^2$. There are also papers \cite{Schmidt_2013} that assume the strong growth condition is satisfied: $\expect{\norms{\nabla f(x, \xi)}^2} \leq \rho \norms{\nabla f(x)}^2.$ In the case where the model is overparameterized \cite{Woodworth_2021_over,Srebro_2010}, it is proposed to estimate the gradient noise as follows: $\expect{\norms{\nabla f(x^*, \xi)}^2} \leq \sigma_*^2$. The essential difference from the previous constraints is that the gradient estimate is evaluated at the solution point and depends directly on the solution of the problem, i.e., if $f^* = \min_x f(x)$ tends to zero, then also $\sigma^2_* \leq L f^*$ decreases. In our work, we use the following constraint on the noise of the biased gradient oracle $\expect{\norms{\gg(x, \xi)}^2} \leq \rho \norms{\nabla f(x)}^2 + \sigma^2$, since this constraint is adaptive and our approach in creating a gradient-free algorithm is based on the approach of the paper \cite{Vaswani_2019}, which also addresses this constraint. The gradient oracle $\gg(x,\xi)$ will be introduced in Subsection \ref{subsec:ASS_gradient_oracle}.   

        \paragraph{Iteration complexity.} The study of a class of convex black-box optimization problems under the condition of higher-order of smoothness began with a 1990 paper\cite{Polyak_1990}, where a method of gradient estimation through the kernel and lower bound estimates of a gradient-free algorithm was proposed. At present, there already exist ``pretty'' results in this direction \cite{Bach_2016,Akhavan_2020,Novitskii_2021,Lobanov_2023,Akhavan_2023}, which can be considered as ``state of the art''. For example, in \cite{Akhavan_2020}, the authors proposed a Zero-Order Stochastic Projected Gradient algorithm that used a central finite difference kernel approximation and required the following iteration $N$ (as well as oracle $T$) complexity to achieve a given accuracy $\varepsilon$: $T = N = \mathcal{\tilde{O}} \left( \frac{d^{2 + \frac{2}{\beta -1}}}{\varepsilon^{2 + \frac{2}{\beta - 1}}} \right).$ In another paper \cite{Novitskii_2021}, the authors managed to improve this dimensionality estimate by using some ``trick'' in analyzing the bias of the gradient-free oracle: $T = N = \mathcal{\tilde{O}} \left( \frac{d^{2 + \frac{1}{\beta -1}}}{\varepsilon^{2 + \frac{2}{\beta - 1}}} \right).$ In a recent paper \cite{Akhavan_2023}, the authors have managed to propose an improved analysis for estimating the bias and second moment (variance) of the gradient approximation, getting rid of the smoothness order dependence in the degree of dimensionality (in the strongly convex case).  Moreover, it is not difficult to show that with the help of this analysis one can get rid of the dependence in the convex case as well (see \cite{Novitskii_2021} for the transformation from the strongly convex case to the convex case). However, these works focus on one of the three optimality criteria, namely oracle complexity. In our work, we use an improved analysis from \cite{Akhavan_2023} for gradient approximation to improve existing estimates of oracle complexity in the convex case, namely getting rid of the dependence of dimensionality on smoothness order, and by using an accelerated version of stochastic gradient descent and working out the batching technique we improve iteration~complexity~estimation.
    
    \subsection{Paper organization}
    This paper has the following structure. Section \ref{sec:Problem_formulation} introduces the formulation of the problem considered in this paper, as well as the main idea of its solution. Section \ref{sec:First_order} presents generalized results for the case of a biased oracle to solve the problem. The main result, which provides a novel gradient-free algorithm, can be found in Section \ref{sec:Main_results}. A discussion of the results is given in Section \ref{sec:Discussion}. Section \ref{sec:Experiments} presents the experiments. While Section \ref{sec:Conclusion} concludes the paper. 

\section{Problem Formulation}
\label{sec:Problem_formulation}
In this section, we introduce the notations, definitions and assumptions used in our analysis to formulate the optimization problem. We also describe the main idea of our approach to solve the black-box optimization problem.

\paragraph{Notation.} We use $\dotprod{x}{y}:= \sum_{i=1}^{d} x_i y_i$ to denote standard inner product of $x,y \in \mathbb{R}^d$, where $x_i$ and $y_i$ are the $i$-th component of $x$ and $y$ respectively. We denote Euclidean norm ($l_2$-norm) in~$\mathbb{R}^d$ as $\norms{x} = \| x\|_2 := \sqrt{\dotprod{x}{x}}$. We use the following notation $B_2^d(r):=\left\{ x \in \mathbb{R}^d : \| x \| \leq r \right\}$ to denote Euclidean ball ($l_2$-ball) and  $S_2^d(r):=\left\{ x \in \mathbb{R}^d : \| x \| = r \right\}$ to denote Euclidean sphere. Operator $\mathbb{E}[\cdot]$ denotes full mathematical expectation. 

We consider a standard optimization problem of the following form, which is commonly encountered in the literature, especially at the first acquaintance~with~optimization methods:

\begin{equation}
    \label{eq:init_problem}
    f^* = \min_{x \in Q \subseteq \mathbb{R}^d} f(x),
\end{equation}
where $f:\mathbb{R}^d \rightarrow \mathbb{R}$ convex function that we want to minimize on the convex set $Q$. This general formulation is a broad class of optimization problems. To narrow down the class of optimization problems, we use a standard formulation of the problem and impose constraints on the function and the gradient oracle in the form of assumptions that will be used in our analysis throughout paper. 

    \subsection{Assumptions on objective function}
    In our analysis presented in Section \ref{sec:First_order}, we assume that function $f$ is $L$-smooth. 
    \begin{assumption}[$L$-smooth]\label{ass:L_smooth}
        Function $f$ is $L$-smooth if it~holds
        \begin{equation*}
            f(y) \leq f(x) + \dotprod{\nabla f(x)}{y - x} + \frac{L}{2} \norms{y - x}^2, \quad \forall x,y \in Q.
        \end{equation*}
    \end{assumption}
    And in Section \ref{sec:Main_results} our theoretical reasoning assumes that the objective function $f(x)$ is not just smooth, but has a higher order of smoothness.
    \begin{assumption}[Higher order smoothness] \label{Ass:Higher_order}
        Let $l$ denote maximal integer number strictly less than~$\beta$. Let $\mathcal{F}_\beta(L)$ denote the set of all functions $f: \mathbb{R}^d \rightarrow \mathbb{R}$ which are differentiable $l$ times and for~all~$x,z \in Q$ the H\"{o}lder-type condition:
        \begin{equation*}
            \left| f(z) - \sum_{0 \leq |n| \leq l} \frac{1}{n!} D^n f(x) (z-x)^n \right| \leq L_\beta \norms{z - x}^\beta,
        \end{equation*}
        where $L_\beta>0$, the sum is over multi-index $n~=~(n_1, ..., n_d) \in \mathbb{N}^d$, we used the notation $n!~=~n_1! \cdots n_d!$, $|n| = n_1 + \cdots + n_d$, and $\forall v = (v_1, ..., v_d) \in \mathbb{R}^d$ we defined $D^n f(x) v^n = \frac{\partial ^{|n|} f(x)}{\partial^{n_1}x_1 \cdots \partial^{n_d}x_d} v_1^{n_1} \cdots v_d^{n_d}$.
    \end{assumption}
    The assumptions introduced in this subsection are standard and common in the literature in related works, e.g., see Assumption \ref{ass:L_smooth} in \cite{Nemirovski_2009,Nesterov_2018}, and Assumption \ref{Ass:Higher_order} in the following works \cite{Polyak_1990,Bach_2016,Akhavan_2023}. Moreover, it is not hard to see the connection between two assumptions, namely, in the case $\beta = 2: L_2 = \frac{L}{2}$.

    \subsection{Assumptions on gradient oracle}
    \label{subsec:ASS_gradient_oracle}
    Before presenting the assumptions on the gradient oracle, we introduce a formal definition, which is used extensively in the analysis for the convergence results of first-order algorithm in Section \ref{sec:First_order}. Regarding convergence of zero-order algorithm, gradient-free oracle will be introduced in Subsection~\ref{subsec:Gradient_approximation}.

    \begin{definition}[Biased Gradient Oracle]
    \label{Definition_1}
        A map $\mathbf{g}~:~\mathbb{R}^d~\times~\mathcal{D} \rightarrow \mathbb{R}^d$ s.t.
        \begin{equation}\label{eq:Definition_1}
            \mathbf{g}(x,\xi) = \nabla f(x, \xi) + \mathbf{b}(x)
        \end{equation}
        for a bias $\mathbf{b}: \mathbb{R}^d \rightarrow \mathbb{R}^d$ and unbiased stochastic gradient $\expect{\nabla f(x, \xi)}=\nabla f(x)$.
    \end{definition}
    We assume that the bias and gradient noise are bounded.
    \begin{assumption}[Bounded bias]
    \label{Ass:Bounded_bias}
        There exists constant $\delta \geq 0$ s.t. $\forall x \in \mathbb{R}^d$
        \begin{equation}\label{eq:Bounded_bias}
            \norms{\mathbf{b}(x)} = \norms{\expect{\gg(x,\xi)} - \nabla f(x)}\leq  \delta.
        \end{equation}
    \end{assumption}
    \begin{assumption}[Bounded noise]
    \label{Ass:Gradient_noise}
        There exists constants $\rho, \sigma^2 \geq 0$ such that the more general condition of strong growth is satisfied $\forall x \in \mathbb{R}^d$
        \begin{equation}\label{eq:Gradient_noise}
            \expect{\norms{\gg(x,\xi)}^2} \leq \rho \norms{\nabla f(x)}^2 + \sigma^2.
        \end{equation}
    \end{assumption}

    Assumptions \ref{Ass:Bounded_bias} and \ref{Ass:Gradient_noise} are not uncommon in works studying optimization algorithms with biased oracle (see Definition \ref{Definition_1}), such as \cite{Ajalloeian_2020,Lobanov_2023,Lobanov_overparametrization}.

    \subsection{The main idea of problem solving}
    The problem presented above does not strongly correspond to the black-box optimization problem, which is also stated in the title of the paper. This is done in order to present in Section \ref{sec:First_order} a first-order algorithm that has access to the noisy value of the gradient (see Definition \ref{Definition_1}). However, our approach to solving the optimization problem \eqref{eq:init_problem} when the gradient is still not available to the algorithm (black-box problems) is to create a gradient-free optimization algorithm based on and exploiting the power of the first-order method. Despite the fact that the original problem \eqref{eq:init_problem} is deterministic, we must rely on a first-order optimization algorithm that solves exactly the stochastic optimization problem $f(x) := \expect{f(x,\xi)}$, since ``stochasticity'' is artificially created in the gradient approximation (see Subsection \ref{subsec:Gradient_approximation}). Moreover, the Kernel approximation is a biased gradient estimator, so it is important to choose an algorithm that accounts for the imprecision in the gradient oracle. Thus, to summarize our approach, in Section \ref{sec:First_order} we generalize the SOTA results to the case with a biased gradient oracle (see Definition \ref{Definition_1}) that satisfies Assumptions \ref{Ass:Bounded_bias} and~\ref{Ass:Gradient_noise}, and use this first-order algorithm to create a gradient-free algorithm (see Section \ref{sec:Main_results}) for solving the black-box optimization problem under the condition of higher-order of smoothness of the objective function (see Assumption \ref{Ass:Higher_order}). 

\section{Generalization of Convergence Results for Accelerated SGD to the Biased Oracle}
\label{sec:First_order}
In this section, we provide the first-order algorithm on which the novel gradient-free method for solving the black-box optimization problem in Section 1 will be based. Since this first-order algorithm must solve a stochastic optimization problem (due to the artificial ``stochasticity'' in the gradient approximation: $\ee \in S_{2}^d(1)$, which will be introduced later), we reformulate the initial optimization problem as follows \eqref{eq:init_problem}:
\begin{equation}
    \label{eq:stoc_init_problem}
    f^* = \min_{x \in Q \subseteq \mathbb{R}^d} \left\{ f(x) := \expect{f(x, \xi)} \right\}.
\end{equation}
Next, before providing the convergence of the first-order algorithm with the biased gradient oracle, we present the known convergence results of accelerated stochastic gradient descent for solving problem \eqref{eq:stoc_init_problem}.

\vspace{-1em}
    \subsection{Background}
    In 2019, the authors of \cite{Vaswani_2019} provided convergence results of Nesterov-accelerated Stochastic Gradient Descent \cite{Nesterov_2012} for the problem when the unbiased gradient oracle $\gg(x, \xi) = \nabla f(x, \xi)$ (see Definition \ref{Definition_1} with $\delta = 0$ in Assumption \ref{Ass:Bounded_bias}) satisfies the strong growth condition (see Assumption \ref{Ass:Gradient_noise}). In particular, this algorithm consists of the following update rules:
    
    \begin{empheq}[box=\fbox]{align*}
        x_{k+1} &= y_k - \eta \gg(y_k, \xi_k)\\
        y_k &= \alpha_k z_k + (1-\alpha_k) x_k \\
        z_{k+1} &= \zeta_k z_k + (1 - \zeta_k) y_k - \gamma_k \eta \gg(y_k, \xi_k).
    \end{empheq}

    And has the following convergence result in the case where $\sigma \neq 0$:
    \begin{lemma}[\cite{Vaswani_2019}, Theorem 1]
    \label{lem:Vaswani}
        Let the function $f$ satisfy Assumption \ref{ass:L_smooth}, and the unbiased gradient oracle $\gg(x, \xi) = \nabla f(x, \xi)$ satisfies the strong growth condition ( Assumption \ref{Ass:Gradient_noise}), then the accelerated Stochastic Gradient Descent by Nesterov with chosen parameters:
        \begin{align*}
            \gamma_{k} = \frac{\tilde{\rho}^{-1} + \sqrt{\tilde{\rho}^{-2} + 4\gamma_{k - 1}^{2}}}{2}; \quad a_{k + 1} = \gamma_k \sqrt{\eta \tilde{\rho}}; \quad \alpha_k = \frac{\gamma_{k} \eta}{\gamma_{k} \eta + a_k^2}; \quad \eta = \frac{1}{\tilde{\rho} L},
        \end{align*} 
        where $\tilde{\rho} = \max\{1, \rho\}$, has the following rate of convergence:
        \begin{equation*}
            \boxed{\expect{f(x_{N})} - f^* \lesssim \frac{\tilde{\rho}^2 L R^2}{N^2} + \frac{N \sigma^2}{L \tilde{\rho}^2}.}
        \end{equation*}
    \end{lemma}
    This result is considered a state of the art for this problem formulation, however, as mentioned earlier, due to the Kernel approximation, which accumulates noise (i.e., has bias), we can't use this algorithm as a basis for creating a gradient-free optimization method. Therefore, in the next subsection, we extend the convergence results of the algorithm (Lemma \ref{lem:Vaswani}) to the case where the gradient oracle $\gg(x, \xi)$ (see Definition \ref{Definition_1}) can return a noisy gradient value $\nabla f(x, \xi) + \bb(x)$, i.e., Assumption \ref{Ass:Bounded_bias} is satisfied. 

    \vspace{-1em}
    \subsection{Accelerated SGD with biased gradients}
    In this subsection, we present the main result of Section \ref{sec:First_order}, namely, we provide a first-order algorithm that we will base on in the next section to create a gradient-free algorithm. To achieve one of the main goals of our work, namely to improve and, if necessary, to obtain an optimal estimate of the iteration complexity $N$ of the gradient-free optimization algorithm, we not only generalize the convergence result of Lemma \ref{lem:Vaswani} to the case with a biased oracle, but also get rid of the constant $\rho$ from the first term, thus improving the estimate by the number of successive iterations of the accelerated Stochastic Gradient Descent. We improve the results of Lemma \ref{lem:Vaswani} in terms of iteration complexity by applying the batching technique. Thus, Accelerated Stochastic Gradient Descent (Nesterov acceleration) with a biased gradient oracle $\gg(x, \xi)$ (see Definition \ref{Definition_1}) has the following convergence results presented in Theorem~\ref{th:First_order}.

    \begin{theorem}[Biased AccSGD]
    \label{th:First_order}
        Let the function $f$ satisfy Assumption \ref{ass:L_smooth}, and the gradient oracle $\gg(x, \xi)$ from Definition \ref{Definition_1} satisfies Assumptions \ref{Ass:Bounded_bias} and \ref{Ass:Gradient_noise}, then the accelerated Stochastic Gradient Descent with batching ($B$ is a batch size) by Nesterov with $\rho_B = \max\{1, \frac{\rho}{B}\}$ and chosen parameters:
        \begin{align*}
            \gamma_{k} = \frac{\rho_B^{-1} + \sqrt{\rho_B^{-2} + 4\gamma_{k - 1}^{2}}}{2}; \quad a_{k + 1} = \gamma_k \sqrt{\eta \rho_B}; \quad \alpha_k = \frac{\gamma_{k} \eta}{\gamma_{k} \eta + a_k^2}; \quad \eta = \frac{1}{\rho_B L}
        \end{align*} 
        has the following rate of convergence:
        \begin{equation*}
            \boxed{\expect{f(x_{N})} - f^* \lesssim \frac{\rho_B^2 L R^2}{N^2} + \frac{N \sigma^2}{\rho_B^2 L B } + \delta \tilde{R} + \frac{N}{L} \delta^2.}
        \end{equation*}
    \end{theorem}
    It is not hard to see that the convergence result of the accelerated batched algorithm presented in Theorem \ref{th:First_order} is a generalization to the case when the gradient oracle $\gg(x, \xi)$ (see Definition \ref{Definition_1}) returns a noisy gradient value. If we put $\delta = 0$ we get exactly the same convergence as in Lemma \ref{lem:Vaswani} up to the constants $\rho$ from the condition of strongly growing (see Assumption \ref{Ass:Gradient_noise}) and $B$ the size of the batch. The two terms accounting for noise accumulation are standard for the accelerated algorithm (see, e.g., \cite{Gorbunov_2019,Dvinskikh_2021,Vasin_2023}) and can be found, for example, by using the ($\delta,L$)-oracle technique \cite{Gasnikov_ICML}. It is through the use of the batched technique in Theorem \ref{th:First_order} that we will able to achieve an optimal estimate on the iteration complexity $N = \mathcal{O}\left(\sqrt{\varepsilon^{-1} LR^2}\right)$ that will obtained from the first term, since it dominates the second term for a sufficiently large value of the batch size $B \geq \rho$. This result allows us to use this accelerated batched algorithm to create a gradient-free method for solving the black-box optimization problem under the condition of higher-order of smoothness of the objective function $f$. A detailed proof of Theorem \ref{th:First_order} can be found~in~Appendix~\ref{app:proof_th1}.

\vspace{-1em}
\section{Main Results}
\label{sec:Main_results}
In this section, we present the main result of our paper, namely a novel gradient-free method for solving the black-box optimization problem \eqref{eq:init_problem} with the condition that the objective function is not only smooth but also has a higher order of smoothness (i.e., the Assumption \ref{Ass:Higher_order} is satisfied). Our approach to create a gradient-free algorithm is to choose and use a gradient estimate $\gg(x, \ee)$ (an approximation of the gradient that will account for the higher-order of smoothness of the function) instead of the real gradient oracle $\gg(x, \xi)$ see Definition \ref{Definition_1} in the accelerated batched first-order method. 

\vspace{-1em}
\subsection{Gradient approximation}
    \label{subsec:Gradient_approximation}
    To solve a deterministic convex black-box optimization problem \eqref{eq:init_problem}, where the ``black box'' plays the role of a gradient-free oracle $\tilde{f}$, which is formally defined as follows: we assume that the oracle $\tilde{f}$ can only return the value of the objective function $f(x)$ at the requested point with some stochastic noise~$\tilde{\xi}$:
    \begin{equation}
        \label{eq:gradient_free_oracle}
        \Tilde{f} = f(x) + \tilde{\xi},
        \vspace{-0.5em}
    \end{equation}
    
    where $\tilde{\xi}$ is stochastic, possibly adversarial, noise, $\expect{\tilde{\xi}^2} \leq \Delta^2$. Then we use the so-called ``Kernel-based approximation'', which was presented in 1990 in the paper \cite{Polyak_1990} and was recognized years later in a number of papers \cite{Bach_2016,Akhavan_2020,Novitskii_2021,Lobanov_2023,Akhavan_2023}, as an approximation of the gradient that takes into account the information about the higher-order of smoothness, has the following form:
    \begin{equation}\label{eq:gradient_approx}
        \gg(x,\ee) = d \frac{f(x+ h r \ee) + \tilde{\xi}_1 - f(x - h r \ee) - \tilde{\xi}_2}{2 h} K(r) \ee,
    \end{equation}
    where $h>0$ is a smoothing parameter, $\ee \in S_2^d(1)$ is a vector uniformly distributed on the Euclidean unit sphere, $r$ is a vector uniformly distributed on the interval $r \in [0,1]$, $K:~[-1,1]~\rightarrow~\mathbb{R}$ is a kernel function that satisfies
        \begin{gather*}
            \mathbb{E}[K(u)] = 0, \; \mathbb{E}[u K(u)] = 1, \;
            \mathbb{E}[u^j K(u)] = 0, \; j=2,...,l, \; \mathbb{E}[|u|^\beta |K(u)|] < \infty.
        \end{gather*}
    This concept of noise is often found in the literature \cite{Akhavan_2020,Lobanov_2023}, where the $\tilde{\xi}_1 \neq \tilde{\xi}_2 $ such that $\mathbb{E}[\tilde{\xi}_1^2] \leq \Delta^2$ and $\mathbb{E}[\tilde{\xi}_2^2] \leq \Delta^2$, $\Delta \geq 0$ is level noise, and the random variables $\tilde{\xi}_1$ and $\tilde{\xi}_2$ are independent from $\ee$ and $r$. Also, this concept does not necessarily have to have a zero mean $\tilde{\xi}_1$ and $\tilde{\xi}_2$. It is enough that $\mathbb{E}[\tilde{\xi}_1 \ee] = 0$ and $\mathbb{E}[\tilde{\xi}_2 \ee] = 0$. Moreover, the gradient approximation \eqref{eq:gradient_approx} may at first glance appear to be an approximation with two-point feedback because of the structure of the central finite difference, but this is not entirely true. Since $\tilde{\xi}_1 \neq \tilde{\xi}_2$ and if we consider $\tilde{\xi}_1$ and $\tilde{\xi}_2$ as concrete realizations of the objective function $f(x)$, it is clear that the function cannot be computed on the same realization twice per iteration. Thus, the approximation with this concept of gradient-free oracle \eqref{eq:gradient_free_oracle} is an approximation with one-point feedback. 

    \subsection{Zero-order accelerated stochastic gradient descent} \label{subsec:ZO_AccSGD}
    Now that we have chosen the gradient approximation and the accelerated batched first-order method, we can present a novel gradient-free algorithm Zero-Order Accelerated Stochastic Gradient Descent (ZO-AccSGD), which is obtained by replacing the real gradient with the gradient approximation \eqref{eq:gradient_approx}.

    \begin{algorithm}[thb]
       \caption{Zero-Order Accelerated Stochastic Gradient Descent}
       \label{algorithm}
        \begin{algorithmic}
           \STATE {\bfseries Input:} iteration number $N$, batch size $B$, Kernel $K: [-1, 1] \rightarrow \mathbb{R}$, step size $\eta$, smoothing parameter $h$, $x_0=y_0=z_0~\in~\mathbb{R}^d$, $\alpha_0 = \gamma_0 = 0$.
           
           \FOR{$k=0$ {\bfseries to} $N-1$}
           \STATE  {1.} ~~~Sample vectors $\ee_1, \ee_2 ..., \ee_B$ uniformly distributed on the unit sphere $S_2^d(1)$ and  
           \STATE  \textcolor{white}{1.} ~~~scalars $r_1, r_2, ..., r_B$ uniformly distributed on the interval [-1, 1] independently
           \STATE  {2.} ~~~Define $\gg(x_k,\ee_i) =  d \frac{\Tilde{f}(x_k+h r_i \ee_i) - \Tilde{f}(x_k - h r_i \ee_i)}{2 h} K(r_i) \ee_i$ via \eqref{eq:gradient_free_oracle}
           \STATE  {3.} ~~~Calculate $\gg_k = \frac{1}{B} \sum_{i=1}^B \gg (x_k,\ee_i)$ 
           \STATE  {4.} ~~~$x_{k+1} \gets y_k - \eta \gg_k$ 
           \STATE {5.} ~~~$z_{k+1} \gets z_k - \gamma_k \eta \gg_k$
           \STATE {6.} ~~~$ y_{k + 1} \gets \alpha_{k + 1} z_{k + 1} + (1-\alpha_{k + 1}) x_{k + 1}$
           \ENDFOR
           \STATE {\bfseries Return:} $x_{N}$
        \end{algorithmic}
    \end{algorithm}
    To obtain the convergence rate of the Algorithm \ref{algorithm}, we need to first estimate the bias $\norms{\expect{\gg(x, \ee)} - \nabla f(x)}$ and second moment (variance) $\expect{\norms{\gg(x, \ee)}^2}$ of the gradient approximation $\gg(x, \ee)$ of \eqref{eq:gradient_approx}. Then, by substituting these estimates into the convergence result of the first-order algorithm we plan to rely on (in our case it is the Biased Accelerated Stochastic Gradient Descent, see Theorem \ref{th:First_order}), \al{in particular instead of $\sigma^2$ (see Assumption~\ref{Ass:Gradient_noise}) from the second term we need to substitute the obtained estimate on the second moment (variance) $\expect{\norms{\gg(x, \ee)}^2}$, and instead of $\delta$ (see Assumption~\ref{Ass:Bounded_bias}) of the third and fourth terms we need to substitute the obtained estimate for the bias $\norms{\expect{\gg(x, \ee)} - \nabla f(x)}$, we get the convergence rate of the novel gradient-free algorithm}. Then, using the bias and second moment estimates for the gradient approximation that takes into account information about the higher order of smoothness (Kernel approximation \eqref{eq:gradient_approx}) presented in \cite{Akhavan_2023} we have the following convergence results for Zero-Order Accelerated Stochastic Gradient Descent. \al{Note that in Theorem~\ref{th:gradient_free} we present convergence results for an arbitrary kernel function satisfying the conditions presented in Subsection~\ref{subsec:Gradient_approximation}.  }
    \begin{theorem}[Convergence results]
        \label{th:gradient_free}
        Let the function $f$ satisfy Assumption \ref{Ass:Higher_order} and the gradient approximation $\gg(x, \ee)$ of \eqref{eq:gradient_approx} satisfies Assumptions \ref{Ass:Bounded_bias} and \ref{Ass:Gradient_noise}, then Zero-Order Accelerated Stochastic Gradient Descent (see Algorithm \ref{algorithm}) with $\rho_B = \max \{ 1, \frac{4 d \kappa}{B} \}$, and with the chosen algorithm parameters:
        \begin{align*}
            \gamma_{k} = \frac{\rho_B^{-1} + \sqrt{\rho_B^{-2} + 4\gamma_{k - 1}^{2}}}{2}; \quad a_{k + 1} = \gamma_k \sqrt{\eta \rho_B}; \quad \alpha_k = \frac{\gamma_{k} \eta}{\gamma_{k} \eta + a_k^2}; \quad \eta = \frac{1}{\rho_B L}
        \end{align*}
        converges to the desired $\varepsilon$ accuracy, $\expect{f(x_N)} - f^* \leq \varepsilon$ 
        \begin{itemize}
        
            \item in the case $B \in \left[1, 4d \kappa \right]$,  $h \lesssim \varepsilon^{3/4}$ and $\beta \geq \frac{7}{3}$ after\\
            \fbox{\begin{minipage}{11 cm}
            \begin{equation*}
                N = \mathcal{O}\left( \sqrt{\frac{\al{d^2} L R^2}{B^2\varepsilon}} \right); \quad \quad \quad T = N \cdot B = \mathcal{O}\left( \sqrt{\frac{d^2 L R^2}{\varepsilon}} \right)
            \end{equation*}
            number of iterations and  gradient-free oracle calls, respectively, at
            \begin{equation*}
                \Delta \lesssim \frac{\al{\varepsilon^{3/2}}}{\sqrt{d}} \quad \textit{maximum noise level;}
            \end{equation*}
                
            
            \end{minipage}}\\

            \item \textit{in the case} $B > 4d \kappa$ and $h \lesssim \varepsilon^{1/(\beta - 1)}$ after\\
            \fbox{\begin{minipage}{11 cm}
            \begin{equation*}
                N = \mathcal{O}\left( \sqrt{\frac{L R^2}{ \varepsilon}} \right); \quad  T = N \cdot B = \max \left\{ \mathcal{O}\left( \sqrt{\frac{d^2 L R^2}{ \varepsilon}} \right), \mathcal{O}\left( \frac{d^2 \Delta^2}{\varepsilon^{2 + \frac{2}{\beta - 1}}} \right)\right\}
            \end{equation*}
            number of iterations and  gradient-free oracle calls, respectively, at
            \begin{equation*}
                \Delta \lesssim  \frac{\varepsilon^{\frac{3 \beta + 1}{4 (\beta - 1)}}}{d} B^{1/2}  \quad \textit{maximum noise level;}
            \end{equation*}
            \end{minipage}}
        \end{itemize}
    \end{theorem}
    From the results of Theorem \ref{th:gradient_free}, it is not difficult to see that at the batch size $B = 1$ the dimensionality factor $d$ comes out in the iteration complexity. The dimensionality can be eliminated by batching, i.e. the larger the batch size $B$, the better the iteration complexity $N$ becomes, in particular, starting from $B = 4 d \kappa$ the iteration complexity completely gets rid of dimensionality and reaches the optimal estimate for the accelerated algorithm. It is worth noting that the maximum noise level when the batch size is $1 \leq B \leq 4d \kappa$, in particular when $\beta \geq \frac{7}{3}$ has, perhaps close to the optimal value for the smooth case (i.e., the case where the Assumption \ref{ass:L_smooth} holds): $\Delta \lesssim d^{-1/2} \varepsilon^{3/2}$. However, this estimate is invariant regardless of the order of smoothness. But there is a way to improve the maximum noise level at which the algorithm is still guaranteed to achieve the desired $\varepsilon$ accuracy. If we take the size of the batches larger than $B > 4 d \kappa$, then the maximum allowable noise level $\Delta$ will improve and, in particular, will depend on the order of smoothness $\beta$. That is, the maximum noise level can be maximized in two ways: taking a larger batch size or using a function with higher-order of smoothness. However, improving the maximum noise level entails a deterioration of the oracle complexity, but has no effect on the iteration complexity $N \sim \varepsilon^{-1/2}$. It is not difficult to see that in any case considered, our results outperform all known results (see subsection Related works), in particular, we improve the iterative complexity estimate, as well as get rid of the dimensionality dependence of oracle complexity, and finally we present estimates of the maximum noise level as a function~of~batch~size. 
    \begin{proof}
        First, we write the bias and the second moment of the gradient approximation \eqref{eq:gradient_approx} from the improved analysis of the paper \cite{Akhavan_2023}.
    \paragraph{Bias of gradient approximation}
    Using the variational representation of the Euclidean norm, and definition of gradient approximation \eqref{eq:gradient_approx} we can write:
    \begin{align}
        \norms{\bb(x)} &= \norms{\expect{\gg(x_k,\xi,\ee)} - \nabla f(x_k)} \nonumber \\
        &= \norms{\frac{d}{2 h} \expect{ \left( \Tilde{f}(x_k + h r \ee) - \Tilde{f}(x_k - h r \ee) \right) K(r) \ee } - \nabla f(x_k)}
        \nonumber \\
        & \overset{\circledOne}{=} \norms{\frac{d}{h}\expect{ f(x_k + h r \ee) K(r) \ee} - \nabla f(x_k)}
        \nonumber \\
        & \overset{\circledTwo}{=}   \norms{\expect{\nabla f(x_k + h r \uu) r K(r)} - \nabla f(x_k)}
        \nonumber \\
        & = \sup_{z \in S_2^d(1)} \expect{\left( \nabla_z f(x_k + h r \uu) - \nabla_z f(x_k)\right) r K(r)}
        \nonumber \\
        &\!\!\! \! \! \! \! \! \overset{\eqref{Taylor_expansion_1}, \eqref{Taylor_expansion_2}}{\leq} \kappa_\beta h^{\beta-1}\frac{L}{(l-1)!} \expect{\norms{u}^{\beta - 1}}
        \nonumber \\
        & \leq \kappa_\beta h^{\beta-1}\frac{L}{(l-1)!} \frac{d}{d+\beta - 1}
        \nonumber \\
        & \lesssim \kappa_\beta L h^{\beta - 1}, \label{eq:proof_bias}
    \end{align}
    where $u \in B_2^d(1)$, $\circledOne =$ the equality is obtained from the fact, namely, distribution of $e$ is symmetric, $\circledTwo =$ the equality is obtained from a version of Stokes’ theorem \cite{Zorich_2016} (see Section 13.3.5, Exercise 14a).

    \paragraph{Bounding second moment of gradient approximation} By definition gradient approximation \eqref{eq:gradient_approx} and Wirtinger-Poincare inequality \eqref{eq:Wirtinger_Poincare} we have
    \begin{align}
        &\expect{\norms{\gg(x_k,\xi,\ee)}^2}  = \frac{d^2}{4 h^2} \expect{\norms{\left(\Tilde{f}(x_k + h r \ee) - \Tilde{f}(x_k - h r \ee)\right) K(r) \ee}^2}
        \nonumber \\
        & = \frac{d^2}{4 h^2} \expect{\left(f(x_k + h r \ee) - f(x_k - h r \ee) + (\tilde{\xi}_1 - \tilde{\xi}_2))\right)^2 K^2(r)} 
        \nonumber \\
        &\! \overset{\eqref{eq:squared_norm_sum}}{\leq} \frac{\kappa d^2}{2 h^2} \left( \expect{\left(f(x_k + h r \ee) - f(x_k - h r \ee)\right)^2} + 2 \Delta^2 \right)
        \nonumber \\
        &\! \! \overset{\eqref{eq:Wirtinger_Poincare}}{\leq} \frac{\kappa d^2}{2 h^2} \left( \frac{h^2}{d} \expect{\norms{ \nabla f(x_k + h r \ee) + \nabla f(x_k - h r \ee)}^2} + 2 \Delta^2 \right) 
        \nonumber \\
        & = \frac{\kappa d^2}{2 h^2} \left( \frac{h^2}{d} \expect{\norms{ \nabla f(x_k + h r \ee) + \nabla f(x_k - h r \ee) \pm 2 \nabla f(x_k)}^2} + 2 \Delta^2 \right)
        \nonumber \\
        & \! \! \overset{\eqref{eq:L_smoothness}}{\leq} \underbrace{4d \kappa}_{\rho} \norms{\nabla f(x_k)}^2 + \underbrace{4 d \kappa  L^2 h^2 + \frac{\kappa d^2 \Delta^2}{h^2}}_{\sigma^2}. \label{eq:proof_variance}
    \end{align}
    We can now explicitly obtain the convergence rate of the novel gradient-free Algorithm \ref{algorithm}: Zero-Order Accelerated Stochastic Gradient Descent by substituting the bias \eqref{eq:proof_bias} and the second moment \eqref{eq:proof_variance} of the gradient approximation \eqref{eq:gradient_approx} into the convergence rate of the first-order algorithm, which we use as the base for creating zero-order algorithm, namely Biased Accelerated Stochastic Gradient Descent (see Theorem~\ref{th:First_order}) with $\rho_B = \max \{1, \frac{4\kappa d}{B} \}$:\\

        \resizebox{\linewidth}{!}{$\boxed{\expect{f(x_{N})} - f^* \lesssim \frac{\rho_B^2 L R^2}{N^2} + \frac{N d \kappa  L^2 h^2}{\rho_B^2 L B} + \frac{N \kappa d^2 \Delta^2}{h^2 \rho_B^2 L B} +  \Tilde{R} \kappa_\beta L_\beta h^{\beta - 1} +  \frac{N \kappa_\beta^2 L_\beta^2 h^{2(\beta - 1)}}{L}.}$}\\

        \vspace{1em}
    To obtain estimates for the iteration number $N$, the total number of gradient-free oracle calls $T$, and the maximum noise level $\Delta$, we consider 4 cases depending on the batch size $B$. 
    
    \vspace{1em}
    \paragraph{$\boxed{\textbf{\it Case 1, when }B = 1}$} we have the following convergence rate:\\

        \resizebox{\linewidth}{!}{$\expect{f(x_{N})} - f^* \lesssim \underbrace{\frac{\kappa^2 d^2 L R^2}{N^2}}_{\circledOne} + \underbrace{\frac{N d \kappa  L^2 h^2}{\kappa^2 d^2 L}}_{\circledTwo} + \underbrace{\frac{N \kappa d^2 \Delta^2}{h^2 \kappa^2 d^2 L}}_{\circledThree} +  \underbrace{\Tilde{R} \kappa_\beta L_\beta h^{\beta - 1}}_{\circledFour} +  \underbrace{\frac{N \kappa_\beta^2 L_\beta^2 h^{2(\beta - 1)}}{L}}_{\circledFive}.$}\\
        
    \textbf{From term $\circledOne$}, we find iteration number $N$ required for Algorithm \ref{algorithm} to achieve $\varepsilon$-accuracy:
    \begin{align}
        \circledOne: \quad \frac{\kappa^2 d^2 L R^2}{N^2} \leq \varepsilon \quad & \Rightarrow \quad N \geq \sqrt{\frac{\kappa^2 d^2 L R^2}{\varepsilon}};
        \nonumber \\
         N &= \mathcal{O}\left( \sqrt{\frac{d^2 L R^2}{\varepsilon}} \right). \label{eq:proof_iterations_1}
    \end{align}
    \textbf{From terms $\circledTwo$, $\circledFour$ and $\circledFive$} we find the smoothing parameter $h$:
    \begin{align*}
        &\circledTwo: \quad \frac{N d \kappa  L^2 h^2}{\kappa^2 d^2 L} \leq 
        \varepsilon \quad \Rightarrow \quad h^2 \overset{\eqref{eq:proof_iterations_1}}{\lesssim} \frac{\kappa^2 d^2  \varepsilon^{3/2}}{\kappa^2 d^2} \quad  \Rightarrow \quad h \lesssim \varepsilon^{3/4} ;
        \nonumber \\
        &\circledFour: \quad \Tilde{R} \kappa_\beta L_\beta h^{\beta - 1} \leq \varepsilon \quad \Rightarrow \quad h \lesssim \varepsilon^{1/(\beta - 1)};
        \nonumber \\
        &\circledFive: \quad \frac{N \kappa_\beta^2 L_\beta^2 h^{2(\beta - 1)}}{L} \leq \varepsilon \quad \Rightarrow \quad h^{2(\beta - 1)} \overset{\eqref{eq:proof_iterations_1}}{\lesssim} d^{-1}\varepsilon^{3/2} \quad h \lesssim \frac{\varepsilon^{\frac{3}{4(\beta - 1)}}}{d^{\frac{1}{2(\beta - 1)}}}.
    \end{align*}
    \begin{itemize}
        \item When $\beta \geq \frac{7}{3}$, we have that $h \lesssim \min \{\varepsilon^{3/4}, \varepsilon^{1/(\beta - 1)}, \frac{\varepsilon^{\frac{3}{4(\beta - 1)}}}{d^{\frac{1}{2(\beta - 1)}}} \}  = \varepsilon^{3/4}$.

        \textbf{From term $\circledThree$}, we find the maximum noise level $\Delta$ at which Algorithm \ref{algorithm} can still achieve the desired accuracy:
        \begin{align*}
            &\circledThree: \quad \frac{N \kappa d^2 \Delta^2}{h^2 \kappa^2 d^2 L} \leq \varepsilon \quad \Rightarrow \quad \Delta^2 \lesssim \frac{\varepsilon^{3}  d^2}{d^3} \quad \Rightarrow \quad \Delta \lesssim \frac{\varepsilon^{3/2}}{\sqrt{d}}.
        \end{align*}

        \item When $\beta < \frac{7}{3}$, we have that $h \lesssim \min \{\varepsilon^{3/4}, \varepsilon^{1/(\beta - 1)}, \frac{\varepsilon^{\frac{3}{4(\beta - 1)}}}{d^{\frac{1}{2(\beta - 1)}}} \}  = \varepsilon^{1/(\beta - 1)}$.

        \textbf{From term $\circledThree$}, we find the maximum noise level $\Delta$ at which Algorithm \ref{algorithm} can still achieve the desired accuracy:
        \begin{align*}
            &\circledThree: \quad \frac{N \kappa d^2 \Delta^2}{h^2 \kappa^2 d^2 L} \leq \varepsilon \quad \Rightarrow \quad \Delta^2 \lesssim \frac{\varepsilon^{\frac{3}{2} + \frac{2}{\beta - 1}}  d^2}{d^3} \quad \Rightarrow \quad \Delta \lesssim \frac{\varepsilon^{\frac{3\beta + 1}{4(\beta - 1)}}}{\sqrt{d}}.
        \end{align*}
    \end{itemize}
    The oracle complexity $T$ in this case coincides with the iteration complexity $N$ and has the following form:
    \begin{equation*}
        T = N \cdot B = \mathcal{O}\left( \sqrt{\frac{d^2 L R^2}{\varepsilon}} \right).
    \end{equation*}

    \paragraph{$\boxed{\textbf{\it Case 2, when }1 < B < 4 \kappa d}$} we have the following convergence rate:\\
    
        \resizebox{\linewidth}{!}{$\expect{f(x_{N})} - f^* \lesssim \underbrace{\frac{\kappa^2 d^2 L R^2}{N^2 B^2}}_{\circledOne} + \underbrace{\frac{N d \kappa  L^2 h^2 B^2}{\kappa^2 d^2 L B}}_{\circledTwo} + \underbrace{\frac{N \kappa d^2 \Delta^2 B^2}{h^2 \kappa^2 d^2 L B}}_{\circledThree}  +  \underbrace{\Tilde{R} \kappa_\beta L_\beta h^{\beta - 1}}_{\circledFour} +  \underbrace{\frac{N \kappa_\beta^2 L_\beta^2 h^{2(\beta - 1)}}{L}}_{\circledFive}.$}\\
        
    \textbf{From term $\circledOne$}, we find iteration number $N$ required for Algorithm \ref{algorithm} to achieve $\varepsilon$-accuracy:
    \begin{align}
        \circledOne: \quad \frac{\kappa^2 d^2 L R^2}{N^2 B^2} \leq \varepsilon \quad & \Rightarrow \quad N \geq \sqrt{\frac{\kappa^2 d^2 L R^2}{\varepsilon B^2}};
        \nonumber \\
         N &= \mathcal{O}\left( \sqrt{\frac{d^2 L R^2}{\varepsilon B^2}} \right). \label{eq:proof_iterations_2}
    \end{align}
    \textbf{From terms $\circledTwo$, $\circledFour$ and $\circledFive$} we find the smoothing parameter $h$:
    \begin{align*}
        &\circledTwo: \quad \frac{N d \kappa  L^2 h^2 B^2}{\kappa^2 d^2 L B} \leq 
        \varepsilon \quad \Rightarrow \quad h^2 \overset{\eqref{eq:proof_iterations_2}}{\lesssim} \frac{\kappa^2 d^2 B^2  \varepsilon^{3/2}}{\kappa^2 d^2 B^2} \quad  \Rightarrow \quad h \lesssim \varepsilon^{3/4} ;
        \nonumber \\
        &\circledFour: \quad \Tilde{R} \kappa_\beta L_\beta h^{\beta - 1} \leq \varepsilon \quad \Rightarrow \quad h \lesssim \varepsilon^{1/(\beta - 1)};
        \nonumber \\
        &\circledFive: \quad \frac{N \kappa_\beta^2 L_\beta^2 h^{2(\beta - 1)}}{L} \leq \varepsilon \quad \Rightarrow \quad h^{2(\beta - 1)} \overset{\eqref{eq:proof_iterations_2}}{\lesssim} d^{-1}\varepsilon^{3/2} \quad h \lesssim \frac{\varepsilon^{\frac{3}{4(\beta - 1)}}}{d^{\frac{1}{2(\beta - 1)}}}.
    \end{align*}
    \begin{itemize}
        \item When $\beta \geq \frac{7}{3}$, we have that $h \lesssim \min \{\varepsilon^{3/4}, \varepsilon^{1/(\beta - 1)}, \frac{\varepsilon^{\frac{3}{4(\beta - 1)}}}{d^{\frac{1}{2(\beta - 1)}}} \}  = \varepsilon^{3/4}$.

        \textbf{From term $\circledThree$}, we find the maximum noise level $\Delta$ at which Algorithm \ref{algorithm} can still achieve the desired accuracy:
        \begin{align*}
            &\circledThree: \quad \frac{N \kappa d^2 B^2 \Delta^2}{h^2 \kappa^2 d^2 L B} \leq \varepsilon \quad \Rightarrow \quad \Delta^2 \lesssim \frac{\varepsilon^{3}  d^2 B^2}{d^3 B^2} \quad \Rightarrow \quad \Delta \lesssim \frac{\varepsilon^{3/2}}{\sqrt{d}}.
        \end{align*}

        \item When $\beta < \frac{7}{3}$, we have that $h \lesssim \min \{\varepsilon^{3/4}, \varepsilon^{1/(\beta - 1)}, \frac{\varepsilon^{\frac{3}{4(\beta - 1)}}}{d^{\frac{1}{2(\beta - 1)}}} \}  = \varepsilon^{1/(\beta - 1)}$.

        \textbf{From term $\circledThree$}, we find the maximum noise level $\Delta$ at which Algorithm \ref{algorithm} can still achieve the desired accuracy:
        \begin{align*}
            &\circledThree: \quad \frac{N \kappa d^2 B^2 \Delta^2}{h^2 \kappa^2 d^2 L B} \leq \varepsilon \quad \Rightarrow \quad \Delta^2 \lesssim \frac{\varepsilon^{\frac{3}{2} + \frac{2}{\beta - 1}}  d^2 B^2}{d^3 B^2} \quad \Rightarrow \quad \Delta \lesssim \frac{\varepsilon^{\frac{3\beta + 1}{4(\beta - 1)}}}{\sqrt{d}}.
        \end{align*}
    \end{itemize}
    The oracle complexity $T$ in this case has the following form:
    \begin{equation*}
        T = N \cdot B = \mathcal{O}\left( \sqrt{\frac{d^2 L R^2}{\varepsilon}} \right).
    \end{equation*}

    \paragraph{$\boxed{\textbf{\it Case 3, when }B = 4\kappa d}$} we have the following convergence rate:
    \begin{equation*}
        \expect{f(x_{N})} - f^* \lesssim \underbrace{\frac{L R^2}{N^2}}_{\circledOne} + \underbrace{\frac{N d \kappa  L^2 h^2}{ L d \kappa}}_{\circledTwo} + \underbrace{\frac{N \kappa d^2 \Delta^2}{h^2 L \kappa d}}_{\circledThree} +  \underbrace{\Tilde{R} \kappa_\beta L_\beta h^{\beta - 1}}_{\circledFour} +  \underbrace{\frac{N \kappa_\beta^2 L_\beta^2 h^{2(\beta - 1)}}{L}}_{\circledFive}.
    \end{equation*}
    \textbf{From term $\circledOne$}, we find iteration number $N$ required for Algorithm \ref{algorithm} to achieve $\varepsilon$-accuracy:
    \begin{align}
        \circledOne: \quad \frac{L R^2}{N^2} \leq \varepsilon \quad & \Rightarrow \quad N \geq \sqrt{\frac{L R^2}{\varepsilon}};
        \nonumber \\
         N &= \mathcal{O}\left( \sqrt{\frac{L R^2}{\varepsilon}} \right). \label{eq:proof_iterations_3}
    \end{align}
    \textbf{From terms $\circledTwo$, $\circledFour$ and $\circledFive$} we find the smoothing parameter $h$:
    \begin{align*}
        &\circledTwo: \quad \frac{N d \kappa  L^2 h^2}{L d \kappa} \leq 
        \varepsilon \quad \Rightarrow \quad h^2 \overset{\eqref{eq:proof_iterations_3}}{\lesssim} \varepsilon^{3/2} \quad  \Rightarrow \quad h \lesssim \varepsilon^{3/4} ;
        \nonumber \\
        &\circledFour: \quad \Tilde{R} \kappa_\beta L_\beta h^{\beta - 1} \leq \varepsilon \quad \Rightarrow \quad h \lesssim \varepsilon^{1/(\beta - 1)};
        \nonumber \\
        &\circledFive: \quad \frac{N \kappa_\beta^2 L_\beta^2 h^{2(\beta - 1)}}{L} \leq \varepsilon \quad \Rightarrow \quad h^{2(\beta - 1)} \overset{\eqref{eq:proof_iterations_3}}{\lesssim} d^{-1}\varepsilon^{3/2} \quad h \lesssim \frac{\varepsilon^{\frac{3}{4(\beta - 1)}}}{d^{\frac{1}{2(\beta - 1)}}}.
    \end{align*}
    \begin{itemize}
        \item When $\beta \geq \frac{7}{3}$, we have that $h \lesssim \min \{\varepsilon^{3/4}, \varepsilon^{1/(\beta - 1)}, \frac{\varepsilon^{\frac{3}{4(\beta - 1)}}}{d^{\frac{1}{2(\beta - 1)}}} \}  = \varepsilon^{3/4}$.

        \textbf{From term $\circledThree$}, we find the maximum noise level $\Delta$ at which Algorithm \ref{algorithm} can still achieve the desired accuracy:
        \begin{align*}
            &\circledThree: \quad \frac{N \kappa d^2 \Delta^2}{h^2 \kappa d L} \leq \varepsilon \quad \Rightarrow \quad \Delta^2 \lesssim \frac{\varepsilon^{3}  d}{d^2} \quad \Rightarrow \quad \Delta \lesssim \frac{\varepsilon^{3/2}}{\sqrt{d}}.
        \end{align*}

        \item When $\beta < \frac{7}{3}$, we have that $h \lesssim \min \{\varepsilon^{3/4}, \varepsilon^{1/(\beta - 1)}, \frac{\varepsilon^{\frac{3}{4(\beta - 1)}}}{d^{\frac{1}{2(\beta - 1)}}} \}  = \varepsilon^{1/(\beta - 1)}$.

        \textbf{From term $\circledThree$}, we find the maximum noise level $\Delta$ at which Algorithm \ref{algorithm} can still achieve the desired accuracy:
        \begin{align*}
            &\circledThree: \quad \frac{N \kappa d^2 \Delta^2}{h^2 \kappa d L} \leq \varepsilon \quad \Rightarrow \quad \Delta^2 \lesssim \frac{\varepsilon^{\frac{3}{2} + \frac{2}{\beta - 1}}  d}{d^2} \quad \Rightarrow \quad \Delta \lesssim \frac{\varepsilon^{\frac{3\beta + 1}{4(\beta - 1)}}}{\sqrt{d}}.
        \end{align*}
    \end{itemize}
    The oracle complexity $T$ in this case has the following form:
    \begin{equation*}
        T = N \cdot B = \mathcal{O}\left( \sqrt{\frac{d^2 L R^2}{\varepsilon}} \right).
    \end{equation*}

    \paragraph{$\boxed{\textbf{\it Case 4, when }B > 4\kappa d}$} we have the following convergence rate:
    \begin{equation*}
        \expect{f(x_{N})} - f^* \lesssim \underbrace{\frac{L R^2}{N^2}}_{\circledOne} + \underbrace{\frac{N d \kappa  L^2 h^2}{ L B}}_{\circledTwo} + \underbrace{\frac{N \kappa d^2 \Delta^2}{h^2 L B}}_{\circledThree} +  \underbrace{\Tilde{R} \kappa_\beta L_\beta h^{\beta - 1}}_{\circledFour} +  \underbrace{\frac{N \kappa_\beta^2 L_\beta^2 h^{2(\beta - 1)}}{L}}_{\circledFive}.
    \end{equation*}
    \textbf{From term $\circledOne$}, we find iteration number $N$ required for Algorithm \ref{algorithm} to achieve $\varepsilon$-accuracy:
    \begin{align}
        \circledOne: \quad \frac{L R^2}{N^2} \leq \varepsilon \quad & \Rightarrow \quad N \geq \sqrt{\frac{L R^2}{\varepsilon}};
        \nonumber \\
         N &= \mathcal{O}\left( \sqrt{\frac{L R^2}{\varepsilon}} \right). \label{eq:proof_iterations_4}
    \end{align}
    \textbf{From terms $\circledTwo$, $\circledFour$ and $\circledFive$} we find the smoothing parameter $h$:
    \begin{align*}
        &\circledTwo: \quad \frac{N d \kappa  L^2 h^2}{L B} \leq 
        \varepsilon \quad \Rightarrow \quad h^2 \overset{\eqref{eq:proof_iterations_4}}{\lesssim} \frac{\varepsilon^{3/2} B}{d} \quad  \Rightarrow \quad h \lesssim \frac{\varepsilon^{3/4} B^{1/2}}{d^{1/2}} ;
        \nonumber \\
        &\circledFour: \quad \Tilde{R} \kappa_\beta L_\beta h^{\beta - 1} \leq \varepsilon \quad \Rightarrow \quad h \lesssim \varepsilon^{1/(\beta - 1)};
        \nonumber \\
        &\circledFive: \quad \frac{N \kappa_\beta^2 L_\beta^2 h^{2(\beta - 1)}}{L} \leq \varepsilon \quad \Rightarrow \quad h^{2(\beta - 1)} \overset{\eqref{eq:proof_iterations_4}}{\lesssim} d^{-1}\varepsilon^{3/2} \quad h \lesssim \frac{\varepsilon^{\frac{3}{4(\beta - 1)}}}{d^{\frac{1}{2(\beta - 1)}}}.
    \end{align*}
    The smoothing parameter can be estimated as $h \lesssim \min \{\varepsilon^{3/4}, \varepsilon^{1/(\beta - 1)}, \frac{\varepsilon^{\frac{3}{4(\beta - 1)}}}{d^{\frac{1}{2(\beta - 1)}}} \}  = \varepsilon^{1/(\beta - 1)}$.
    \textbf{From term $\circledThree$}, we find the maximum noise level $\Delta$ (via batch size $B$) at which Algorithm~\ref{algorithm} can still achieve the desired accuracy:
    \begin{align*}
        &\circledThree: \quad \frac{N \kappa d^2 \Delta^2}{h^2 L B} \leq \varepsilon \quad \Rightarrow \quad \Delta^2 \lesssim \frac{\varepsilon^{\frac{3}{2} + \frac{2}{\beta - 1}}  B}{d^2} \quad \Rightarrow \quad \Delta \lesssim \frac{\varepsilon^{\frac{3\beta + 1}{4(\beta - 1)}} B^{1/2}}{d}
    \end{align*}
    or let's represent the batch size $B$ via the maximum noise level $\Delta$:
    \begin{align*}
        &\circledThree: \quad \frac{N \kappa d^2 \Delta^2}{h^2 L B} \leq \varepsilon \quad \Rightarrow \quad B \overset{\eqref{eq:proof_iterations_4}}{\gtrsim} \frac{\kappa d^2 \Delta^2}{\varepsilon^{\frac{3}{2} + \frac{2}{\beta - 1}}} \quad \Rightarrow \quad B = \mathcal{O}\left( \frac{d^2 \Delta^2}{\varepsilon^{\frac{3}{2} + \frac{2}{\beta - 1}}} \right).
    \end{align*}
    Then the oracle complexity $T$ in this case has the following form: 
    \begin{equation*}
        T = N \cdot B = \max \left\{\mathcal{O}\left( \sqrt{\frac{d^2 L R^2}{\varepsilon}} \right), \mathcal{O}\left( \frac{d^2 \Delta^2}{\varepsilon^{2 + \frac{2}{\beta - 1}}} \right) \right\}.
    \end{equation*}
    \qed
    \end{proof}

\section{Discussion and Further Work}
\label{sec:Discussion}
Section \ref{sec:Main_results} focuses on solving the convex deterministic black-box optimization problem \eqref{eq:init_problem}, however, when constructing the gradient-free algorithm (see Subsection \ref{subsec:ZO_AccSGD}) is based on a first-order method that solves the convex stochastic optimization problem \eqref{eq:stoc_init_problem} due to the arising of artificial ``stochasticity'' in the gradient approximation \eqref{eq:gradient_approx}. It is not difficult to show that the results of Theorem \ref{th:gradient_free} will be robust if the original problem of Section \ref{sec:Main_results} is replaced by a stochastic black-box optimization problem, since there will already be two stochasticities in the analysis that can be formally combined into one~$(\tilde{\xi}, \ee)$. 

If we pay attention to the results presented in the works of \cite{Bach_2016,Akhavan_2020,Novitskii_2021} and others, we can see that they ``struggle'' for oracle complexity $T$. However, in high dimensional problems, it is important to be able to distribute the computational power loads, thereby reducing the time taken to solve a particular problem. Therefore, with the help of a not tricky technique, namely with the help of batching technique and using the accelerated algorithm as a base (in particular, Accelerated Stochastic Gradient Descent with accelerated of Nesterov, see Theorem \ref{th:First_order}), we managed to improve the estimate of the number of consecutive iterations $N \sim \varepsilon^{-1/2}$ to achieve the desired accuracy $\varepsilon$ of the solution of the original problem, without worsening the oracle complexity $T$, and moreover improving in terms of dimensionality $d$ for a class of convex optimization problems. It is due to this fact, namely the ability to improve one optimality criterion without compromising the second one, that recently authors of works on gradient-free optimization algorithms have been evaluating the efficiency of their algorithms by three optimality criteria at~once~\cite{Gasnikov_2022}: iteration complexity, total number of calls to the gradient-free oracle, and maximum noise level $\Delta$ at which it is still possible to achieve~desired~accuracy.

\al{Theorem~\ref{th:gradient_free} clearly demonstrates the advantage of Algorithm~\ref{algorithm} (which uses enhanced smoothness information) over existing gradient-free algorithms that do not use this information. For example, when using the proposed algorithm (in the case, for example, when the smoothness of the function $\beta = 5$ and the batch size $B=4d\kappa$) we will converge to the same error floor as the existing algorithms (solving smooth problems $\beta = 2$), only faster. However, we want our algorithm not only to converge faster, but also more accurately (to achieve a better error floor). This is exactly what we have achieved (which is really not trivial). We show that using $B > 4d\kappa$ the maximum noise level (or at a fixed noise level, error floor) accounts for smoothness order and batching costs, preserving optimality on iteration complexity.}

We see the following directions as the development of this work: obtaining convergence results for a $\mu$-strongly convex black-box optimization problem. For this formulation of the problem there are already some results presented in~\cite{Akhavan_2023}, we expect that using similar reasoning, namely generalizing the convergence results of the accelerated algorithm \cite{Vaswani_2019} for solving a strongly convex stochastic optimization problem to the case with a biased gradient oracle (see Definition \ref{Definition_1}) and using the kernel approximation \eqref{eq:gradient_approx}, we will be able to improve the performance of \cite{Akhavan_2023} in terms of iteration complexity $N$, achieving the same oracle complexity estimates $T$, and provide an explicit condition on the maximum noise level $\Delta$. Another direction of development of our work is to improve oracle complexity $T$ for convex and strongly convex black-box optimization problem. It is worth noting that in the class of convex functions we managed to improve oracle complexity, but this upper bound does not match the lower bound presented in \cite{Akhavan_2020,Novitskii_2021}. Finally, the last direction that looks promising at the moment is the study of the maximum allowable noise level. In this paper, we have provided a maximum noise level at which convergence to the desired accuracy $\varepsilon$ is guaranteed; however, we have not guaranteed the optimality of this estimate since the upper bound on the noise level is not yet known. Furthermore, we expect that this estimator can be improved by using a different concept of a gradient-free oracle, in particular when the oracle can output the objective function value with some bounded adversarial deterministic noise (see \cite{Dvinskikh_2022} for details). In this case, we can also expect an improvement in oracle complexity, since the gradient approximation with a central finite difference structure will already have access to two-point feedback.

\vspace{-1em}
\section{Experiments}
\label{sec:Experiments}

In this section, we verify the performance of the proposed algorithm in Section~\ref{sec:Main_results} on model functions as well as on functions of interest in machine learning. In all experiments as the Kernel $K(r)$ of gradient approximation~\eqref{eq:gradient_approx}  we use the already standard function, namely Legendre polynomials, for which it was proved in the paper \cite{Bach_2016} that the constants $\kappa$ and $\kappa_\beta$ do not depend on the dimensionality $d$, but only on the smoothness order~of~$\beta$. We have the following values for different~$\beta$: 
\begin{align*}
    K(r) &= \frac{15r}{4}(5 - 7r^2) & \text{for } \beta = 3, 4;\\
    K(r) &= \frac{195r}{16}(99r^4 - 126r^2 + 35) & \text{for } \beta = 5, 6.
\end{align*}

\subsection{System of linear equations}
We consider solving a system $p$ of linear equations, where the problem \eqref{eq:init_problem} looks like $\min_{x \in \mathbb{R}^d} := \left\|  Ax - b \right\|^2$, where $x \in \mathbb{R}^d$, $A \in \mathbb{R}^{p \times d}$ and $b \in \mathbb{R}^p$. Figure~\ref{fig:Approximations} shows the advantage of the Kernel approximation for the case when the function has a higher order of smoothness. Here we optimize $f(x)$ with the parameters: $d = 64$ (dimension of the problem); $B = 50$ (batch size); $\Delta = 10^{-5}$ (noise level), $\eta = 0.02$ (step-size); $h = 0.5$ (smoothing parameter). Figure~\ref{fig:Comparison_quadratic} shows the advantage of the algorithm proposed in Section~\ref{sec:Main_results} compared to existing accelerated algorithms: ARDFDS from \cite{Gorbunov_2022}, ZO-VARAG from \cite{Chen_2020}. We can see that ZO-VARAG hits the asymptote rather quickly, since the work of \cite{Chen_2020} did not consider the problem formulation with adversarial noise, i.e., the algorithm is not adaptive to noise. It is also worth noting that our algorithm outperforms ARDFDS from \cite{Gorbunov_2022} in particular because ZO-AccSGD uses information about the high smoothness of the function.

\begin{figure}
    \centering
    \begin{subfigure}[b]{0.45\textwidth}
        \includegraphics[width=\textwidth]{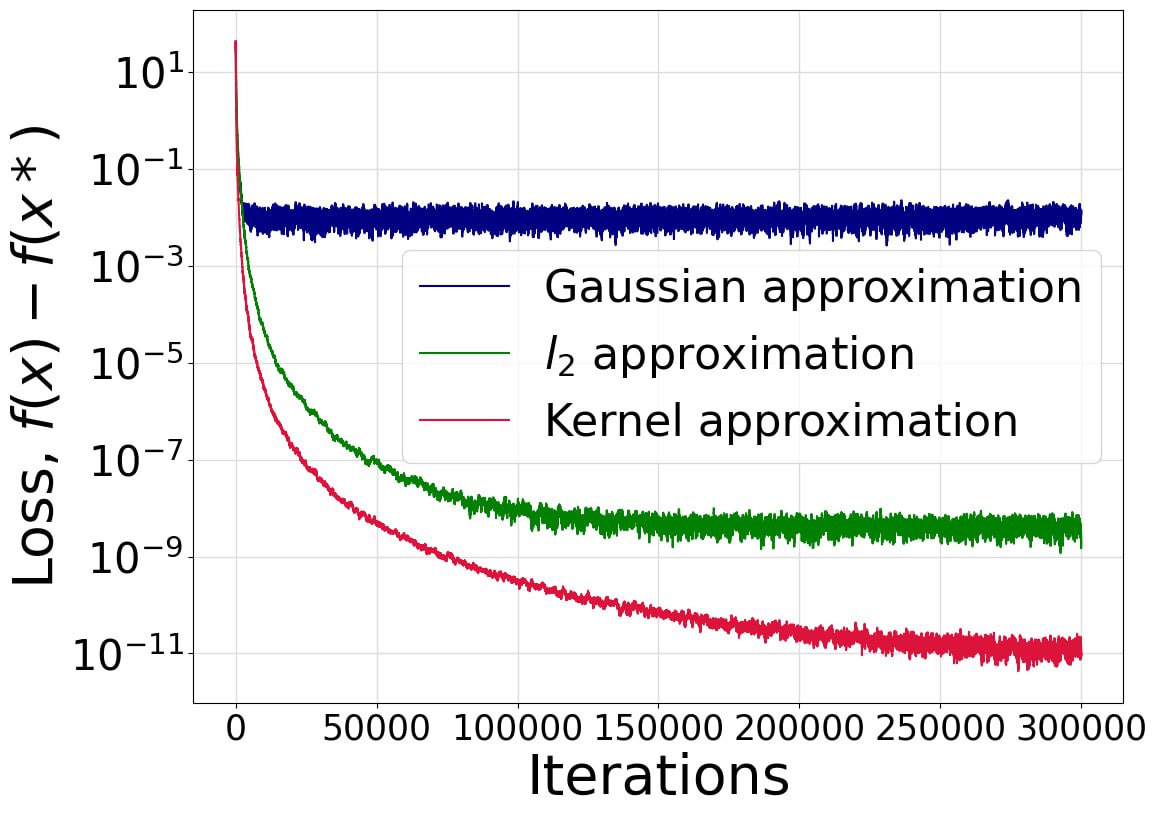}
        \caption{Effect of gradient approximations on the convergence}
        \label{fig:Approximations}
    \end{subfigure}
    ~ 
    \begin{subfigure}[b]{0.45\textwidth}
        \includegraphics[width=\textwidth]{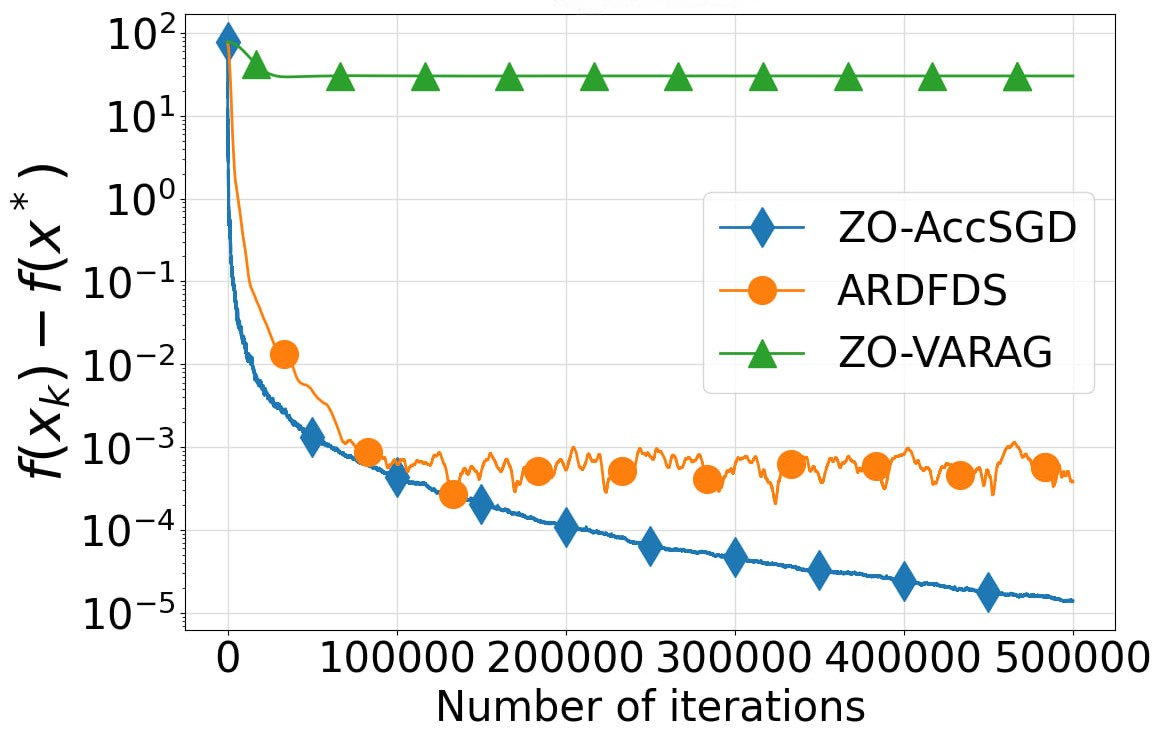}
        \caption{Comparison of convergence of accelerated algorithms}
        \label{fig:Comparison_quadratic}
    \end{subfigure}
    \caption{Numerical experiments}\label{fig:animals}
\end{figure}

From Figure~\ref{fig:ovb_squared_norm} we can see the same convergence rates for the first 20-30 iterates with $B = 4d \kappa = 9000$ and $B = 12000$, due to the same values of $\rho_B$ (optimal in terms of iteration complexity, see Theorem~\ref{th:gradient_free}). However, going further through iterations, the error floor separation becomes more clear. This validates our theoretical results: the overbatching effect does improve the distance to the solution $x^*$ of the proposed Algorithm \ref{algorithm}.

\begin{figure}[H]
    \centering
    \includegraphics[width=0.75\textwidth]{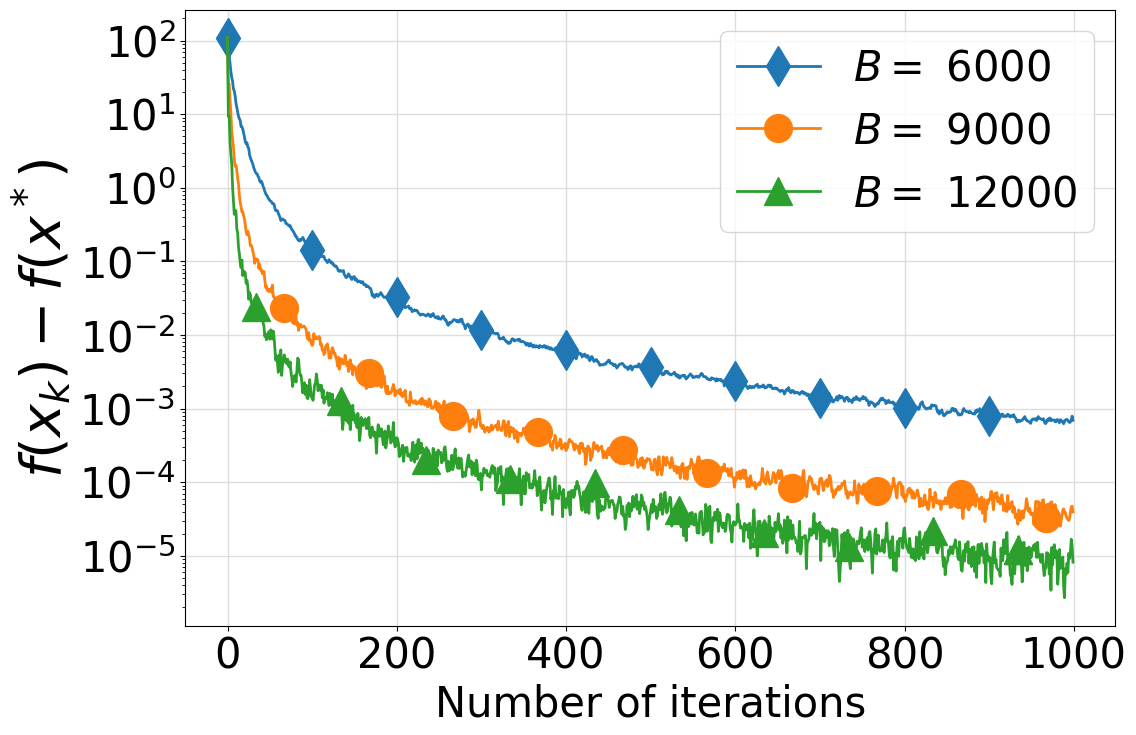} 
    \caption{Demonstration of the overbatching effect on a model problem of solving a linear system of $p$ equations}
    \label{fig:ovb_squared_norm}
\end{figure}

\subsection{Logistic regression}

This Subsection is devoted to numerical experiments on the optimization of a function of interest in machine learning, namely logistic regression:
\begin{equation}
\label{problem:log_reg}
    \min \limits_{x \in \mathbb{R}^d} f(x) = \frac{1}{M} \sum \limits_{i = 1}^M f_i(x),
\end{equation}
where $f_i(x) = \log \big(1 + \exp(-y_i \cdot (Ax)_i)\big)$ stands for the loss on the $i$-th data point, $A \in \mathbb{R}^{M \times d}$ is an instances matrix, $y \in \{-1, 1\}^{M}$ is a label vector and $x \in \mathbb{R}^d$ is a vector of weights. It is easy to show, that logistic regression function is $L$-smooth with $L = \frac{1}{4M} \sqrt{\lambda_{\max} (A^TA)}$, where $\lambda_{\max} (A^TA)$ denotes the largest eigenvalue of the matrix $A^TA$. 

In our experiments we use data from LIBSVM\cite{LIBSVM} library, specifically such datasets as \textit{phishing}, \textit{diabetes} and \textit{hearts}. The information about them can be found in the Table \ref{tab:data}:
\begin{table}[H]
    \centering
    \begin{tabular}{|c|c|c|c|}
        \hline
         & \textit{phishing} & \textit{diabetes} & \textit{hearts} \\
         \hline
        Size ($M$) & 11055 & 768 & 270 \\
        \hline
        Dimension ($d$) & 68 & 8 & 13 \\
        \hline
    \end{tabular}
    \caption{\vspace{0.5 em} Summary of used datasets}
    \label{tab:data}
\end{table}

\vspace{-2em}
In all tests we use standard solvers from scipy library to get an accurate approximation of the solution $f(x^*)$. Next, we set $x_0$ such that $f(x_0) - f(x^*) \sim 1$, and in all experiments smoothing parameter $h$ is equal to $10^{-10}$. We compare the results of our Algorithm~\ref{algorithm} (ZO-AccSGD) with two methods presented in \cite{Gorbunov_2019}: RDFDS and its accelerated version ARDFDS. To find better learning rates we grid searched it for all methods and then took the best parameters for each method. The results for different batching constant $B$ are presented in Figure \ref{fig:additional_exp}:

\begin{figure}[h!]

\begin{subfigure}{0.45\textwidth}
    \centering
    \includegraphics[width=\textwidth]{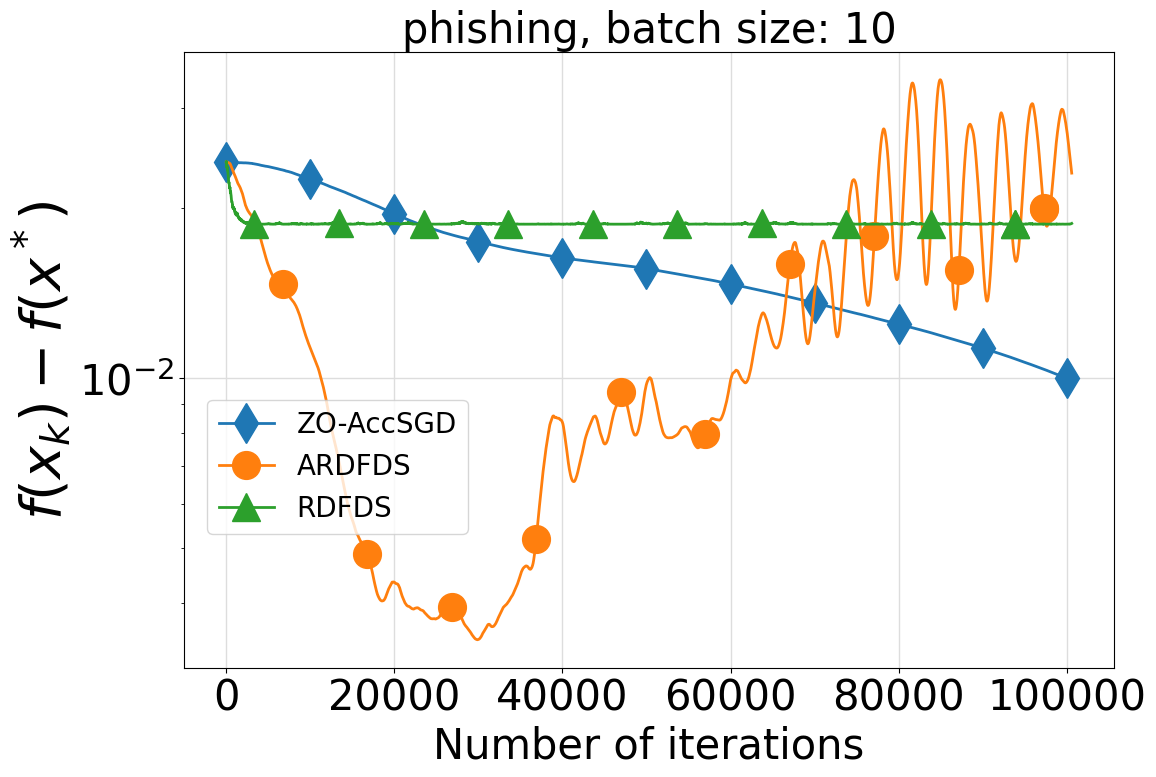}
\end{subfigure}
\hfill
\begin{subfigure}{0.45\textwidth}
    \centering
    \includegraphics[width=\textwidth]{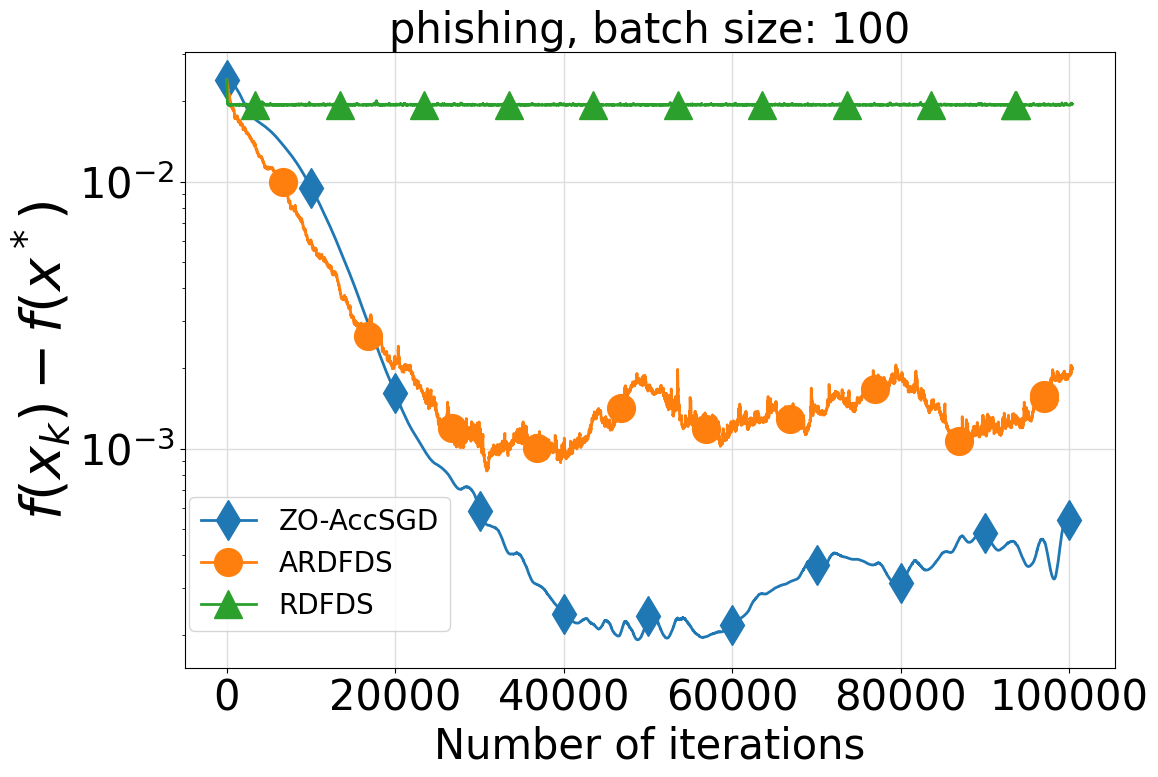}
\end{subfigure}
\begin{subfigure}{0.3\textwidth}
    \centering
    \includegraphics[width=\textwidth]{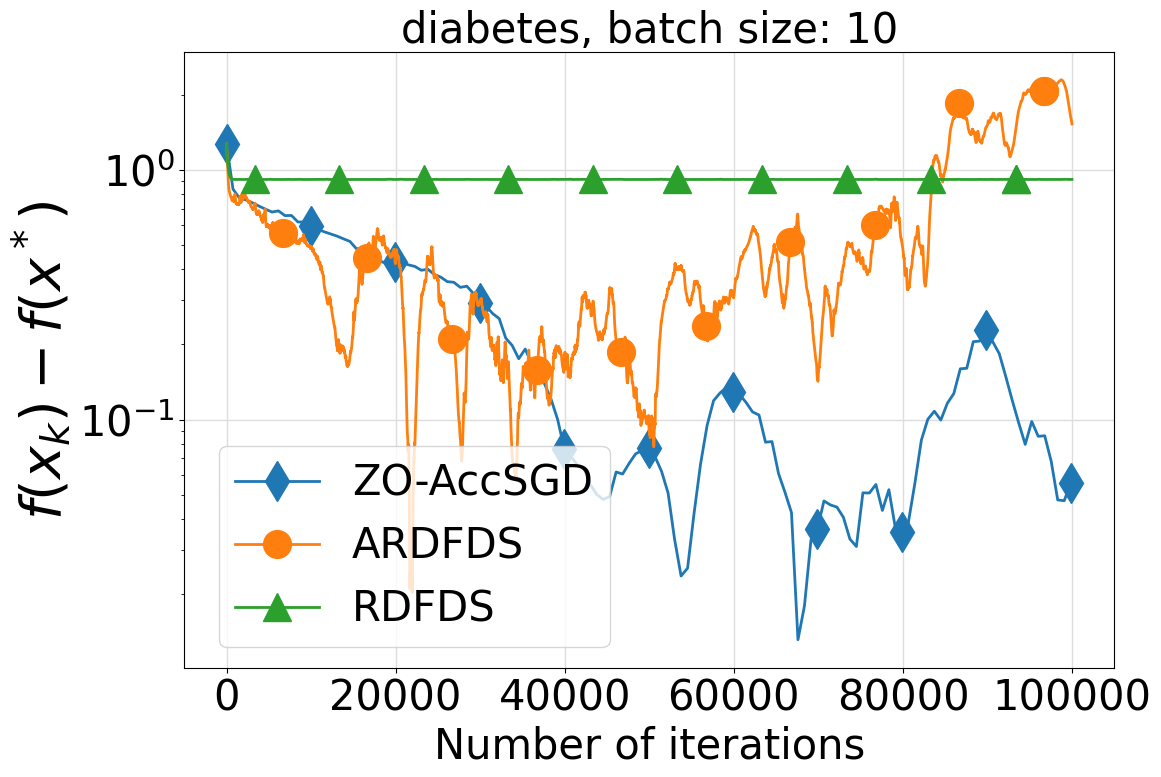}
\end{subfigure}
\hfill
\begin{subfigure}{0.3\textwidth}
    \centering
    \includegraphics[width=\textwidth]{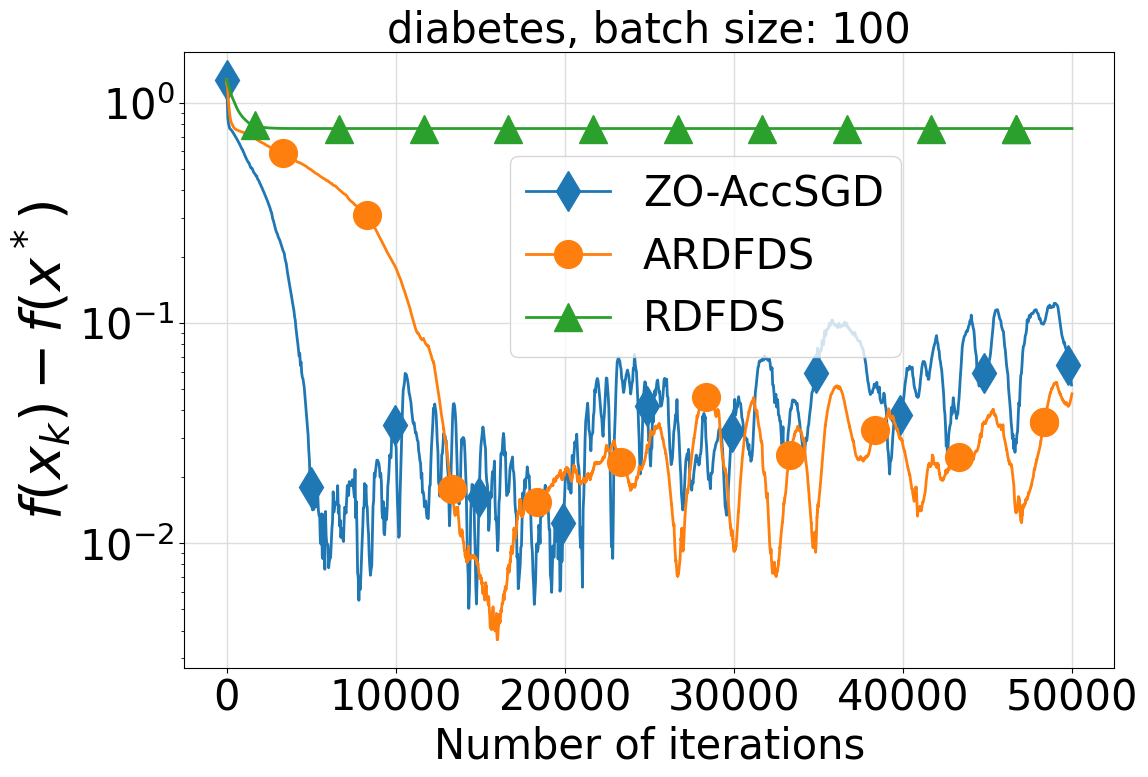}
\end{subfigure}
\hfill
\begin{subfigure}{0.3\textwidth}
    \includegraphics[width=\textwidth]{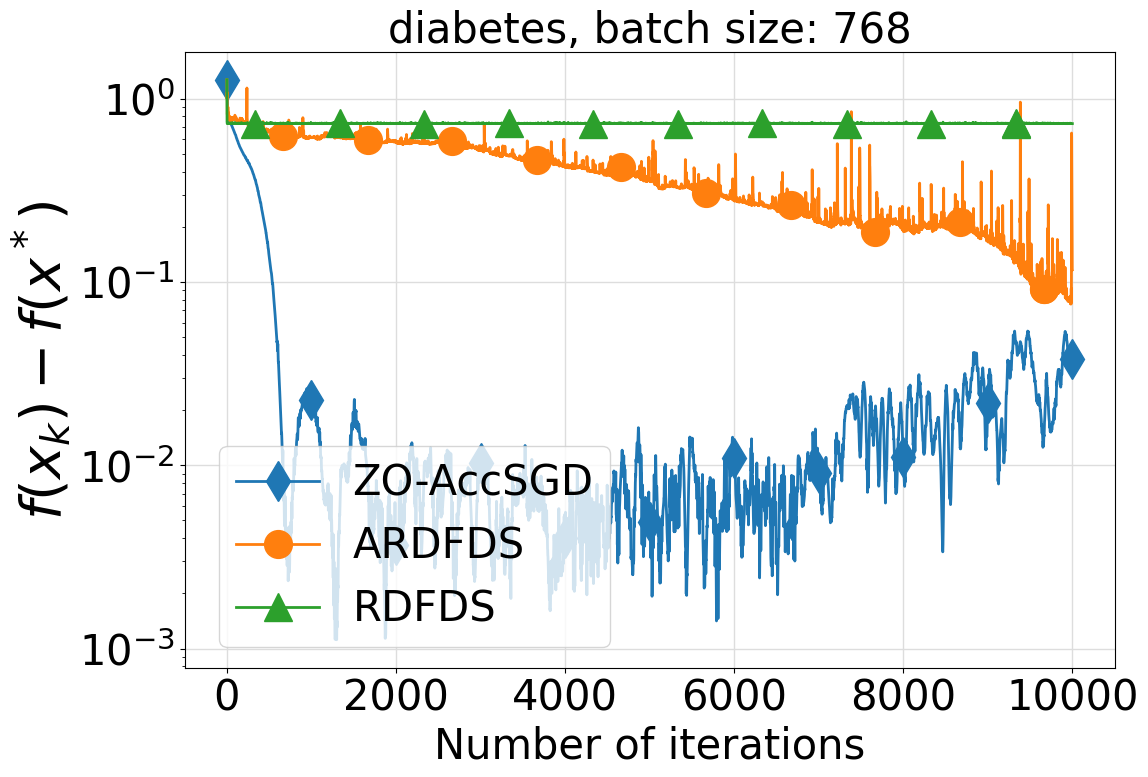}
\end{subfigure}
\begin{subfigure}{0.3\textwidth}
    \centering
    \includegraphics[width=\textwidth]{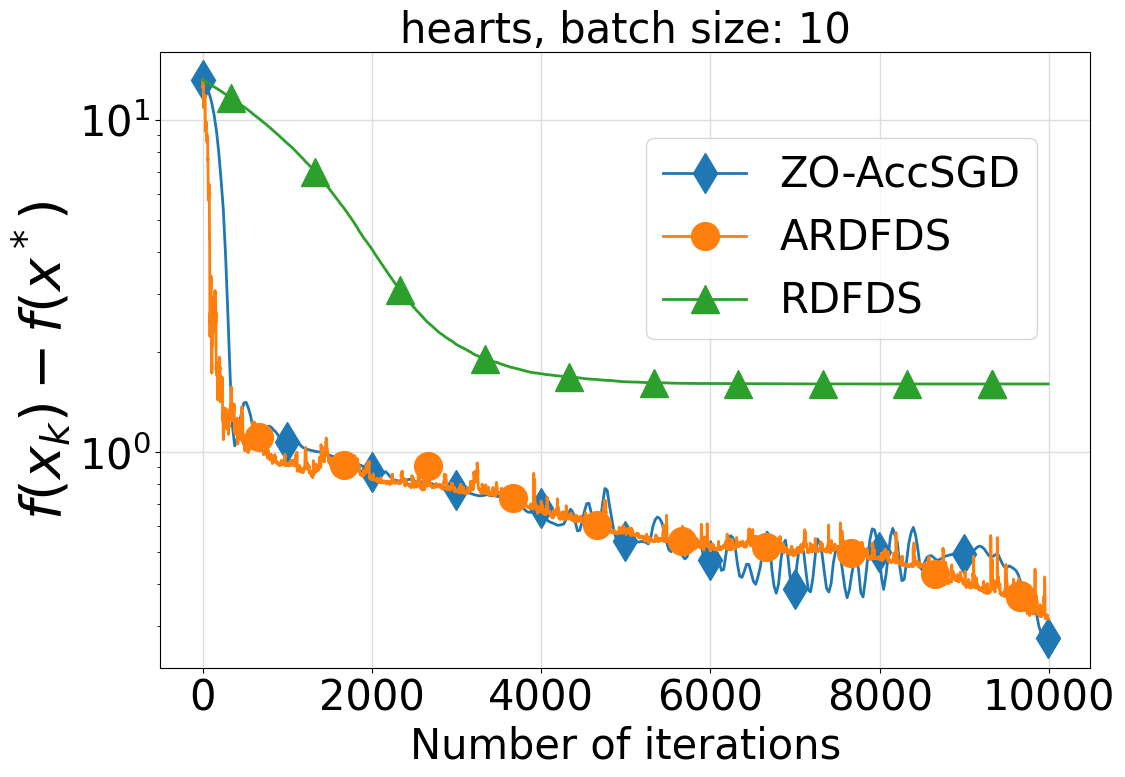}
\end{subfigure}
\hfill
\begin{subfigure}{0.3\textwidth}
    \centering
    \includegraphics[width=\textwidth]{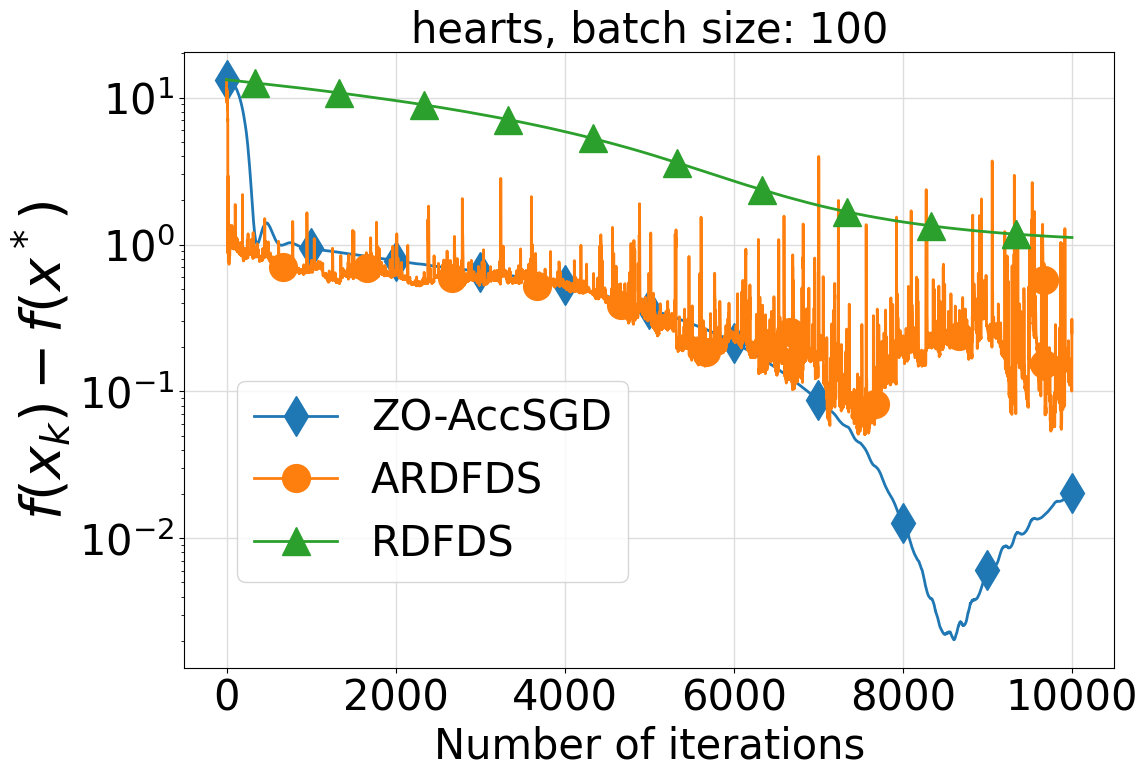}
\end{subfigure}
\hfill
\begin{subfigure}{0.3\textwidth}
    \includegraphics[width=\textwidth]{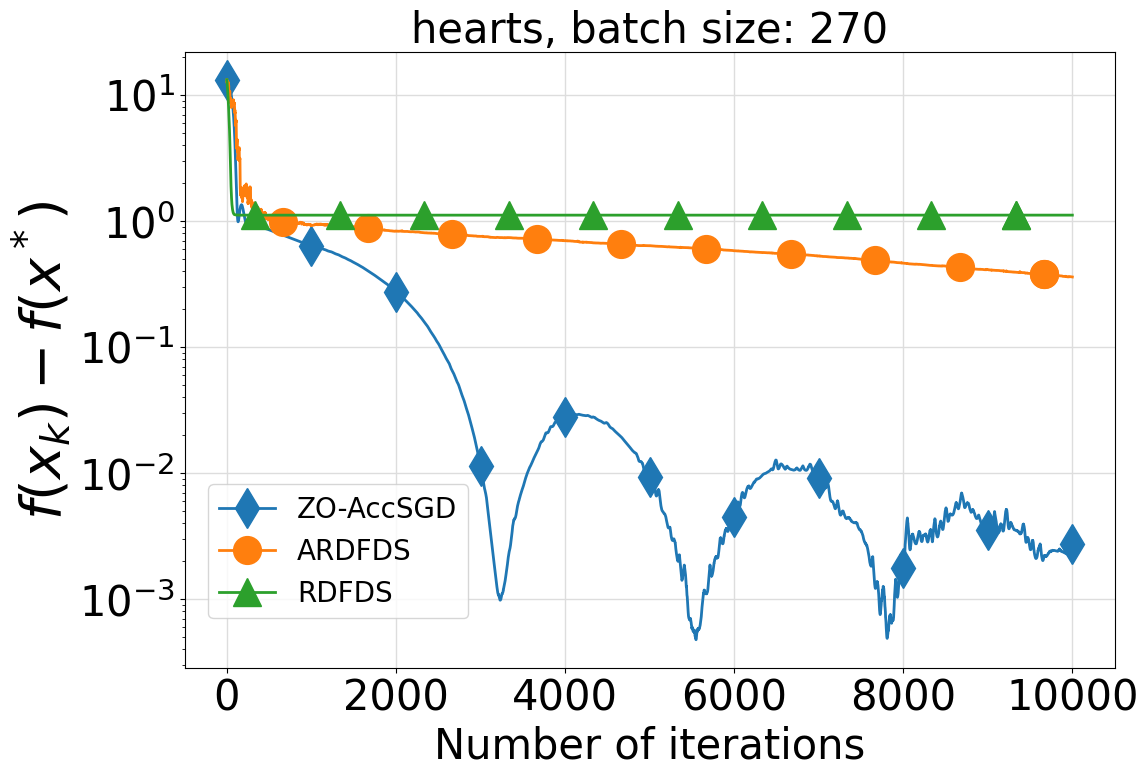}
\end{subfigure}
        
\caption{Numerical results for solving logistic regression problem \eqref{problem:log_reg} for different datasets}
\label{fig:additional_exp}

\end{figure}

As an accelerated method, our Algorithm \ref{algorithm} is sensible to stochasticity, however compared to ARDFDS has a better robustness. Also, it is clear from graphs that bigger batchsize implies lower error floor and faster convergence rate. Even with gridsearched parameters for stepsize in all methods, ZO-AccSGD (see Algorithm \ref{algorithm}) performance is better compared to RDFDS and ARDFDS ones.

\paragraph{Technical aspects.} All experiments were written on Python and are available in anonymous repository at the following link: \url{https://github.com/MetalistForever/ZO-AccSGD}. They were run on the local machine with mobile CPU Ryzen 4800H. Average time for single run with number of iterations $N = 10^5$ and batchsize $B = 10$ on the \textit{hearts} dataset is 9 min ($\approx$ 205 it/s), nonetheless the are more time-consuming tests, for example to run optimizer on \textit{phishing} dataset with $N = 10^5$ and $B = 100$ it took more than an hour.

\vspace{-1em}
\section{Conclusion}
\label{sec:Conclusion}
In this paper, we proposed a novel gradient-free algorithm (see Algorithm \ref{algorithm}) that improves the iteration $N$, oracle complexities $T$, and explicitly defines the maximum noise level $\Delta$ at which the algorithm can still be guaranteed to converge to the desired $\varepsilon$ accuracy in a class of convex black-box optimization problem where the function is not just smooth but has a higher order of smoothness. Our approach for the gradient-free algorithm was based on the work of \cite{Vaswani_2019}, however, due to the biased gradient approximation (Kernel approximation), we generalized the convergence results of this work to the gradient oracle with bias, and applied a batting technique to improve the first term in the convergence of \cite{Vaswani_2019} (this result may be of independent interest). In the experiments section, we confirmed our theoretical results obtained in this paper. In addition, we considered possible developments of this paper.

\begin{acknowledgements}
The work of Alexander Gasnikov, Aleksandr Lobanov  was supported by a grant for research centers in the field of artificial intelligence, provided by the Analytical Center for the Government of the Russian Federation in accordance with the subsidy agreement (agreement identifier 000000D730321P5Q0002) and the agreement with the Ivannikov Institute for System Programming of the Russian Academy of Sciences dated November 2, 2021 No. 70-2021-00142.
\end{acknowledgements}


\bibliographystyle{spmpsci}
{
\normalsize
\bibliography{3_references}
}

\newpage

\appendix  
    \begin{center}
        \LARGE \bf APPENDIX
    \end{center}
    
\section{Auxiliary Facts and Results}

    In this section we list auxiliary facts and results that we use several times in our~proofs.
    
    \subsection{Squared norm of the sum} For all $a_1,...,a_n \in \mathbb{R}^d$, where $n=\{2,3\}$
    \begin{equation}
        \label{eq:squared_norm_sum}
        \norms{a_1 + ... + a_n }^2 \leq n \norms{ a_1 }^2 + ... + n \norms{a_n}^2.
    \end{equation}

    \subsection{Fenchel-Young inequality} For all $a,b\in\mathbb{R}^d$ and $\lambda > 0$
    \begin{equation}
        \dotprod{a}{b} \leq \frac{\|a\|_2^2}{2\lambda} + \frac{\lambda\|b\|_2^2}{2}.\label{eq:fenchel_young_inequality}
    \end{equation}

    \subsection{$L$ smoothness function}
        Function $f$ is called $L$-smooth on $\mathbb{R}^d$ with $L~>~0$ when it is differentiable and its gradient is $L$-Lipschitz continuous on $\mathbb{R}^d$, i.e.\ 
        \begin{equation}
            \norms{\nabla f(x) - \nabla f(y)} \leq L \norms{x - y},\quad \forall x,y\in \mathbb{R}^d. \label{eq:L_smoothness}
        \end{equation}
         It is well-known that $L$-smoothness implies (see e.g., Assumption \ref{ass:L_smooth})
        \begin{eqnarray*}
            f(y) \leq f(x) + \dotprod{\nabla f(x)}{y-x} + \frac{L}{2}\norms{y-x}^2\quad \forall x,y\in \mathbb{R}^d,
        \end{eqnarray*} 
        and if $f$ is additionally convex, then
        \begin{eqnarray*}
            \norms{ \nabla f(x) - \nabla f(y) }^2 \leq 2L \left( f(x) - f(y) - \dotprod{ \nabla f(y)}{x-y} \right) \quad \forall x,y \in \mathbb{R}^d. 
        \end{eqnarray*}

    \subsection{Wirtinger-Poincare inequality}
        Let $f$ is differentiable, then for all $x \in \mathbb{R}^d$, $h \ee \in S^d_2(h)$:
        \begin{equation}\label{eq:Wirtinger_Poincare}
            \expect{f(x+ h \ee)^2} \leq \frac{h^2}{d} \expect{\norms{\nabla f(x + h \ee)}^2}.
        \end{equation}

    \subsection{Taylor expansion} Using the Taylor expansion we have
    \begin{equation}
        \label{Taylor_expansion_1}
        \nabla_z f(x+ h r \uu) = \nabla_z f(x) + \sum_{1 \leq |n| \leq l-1} \frac{(r h)^{|n|}}{n!} D^{(n)} \nabla_z f(x) \uu^n + R(h r \uu),
    \end{equation}
    where by assumption 
    \begin{equation}\label{Taylor_expansion_2}
        |R(h r \uu)| \leq \frac{L}{(l-1)!}  \norms{h r \uu}^{\beta - 1} = \frac{L}{(l-1)!}|r|^{\beta - 1} h^{\beta - 1} \norms{\uu}^{\beta - 1}.
    \end{equation}
    
    \subsection{Kernel property} If $\ee$ is uniformly distributed on $S_2^d(1)$ we have $\mathbb{E}[\ee \ee^{\text{T}}] = (1/d)I_{d \times d}$, where $I_{d \times d}$ is the identity matrix. Therefore, using the facts $\mathbb{E}[r K(r)] = 1$ and $\mathbb{E}[r^{|n|} K(r)] = 0$ for $2 \leq |n| \leq l$ we have
        \begin{equation}\label{Kernel_property}
            \mathbb{E}\left[ \frac{d}{h} \left( \dotprod{\nabla f(x)}{h r \ee} + \sum_{2 \leq |n| \leq l} \frac{(r h)^{|n|}}{n!} D^{(n)} f(x) \ee^n \right) K(r) \ee  \right] = \nabla f(x).
        \end{equation}

    \subsection{Bounds of the Weighted Sum of Legendre Polynomials}
    Let $\kappa_\beta = \int |u|^\beta |K(u)| du$ and set $\kappa = \int K^2(u) du$. Then if $K$ be a weighted sum of Legendre polynomials, then it is proved in (see Appendix A.3, \cite{Bach_2016}) that $\kappa_\beta$ and $\kappa$ do not depend on $d$, they~depend~only~on~$\beta$, such that for $\beta \geq 1$: 
        \begin{equation}\label{eq_remark_1}
            \kappa_\beta \leq 2 \sqrt{2} (\beta-1),
        \end{equation}
        \begin{equation}\label{eq_remark_2}
            \kappa \leq 3 \beta^{3}.
        \end{equation}

\section{Proof of Theorem \ref{th:First_order}}\label{app:proof_th1}
In this section, we present a detailed description of the derivation of the results of Theorem \ref{th:First_order}, which generalize the result of Lemma \ref{lem:Vaswani} to the case of a biased gradient oracle (see Definition \ref{Definition_1}). Therefore, our analysis will rely on the proof of Theorem 1 of \cite{Vaswani_2019}, working through the summands responsible for the accumulation of noise in the gradient oracle. Before starting the proof, we recall that the update equations for SGD with Nesterov acceleration have the following general form:
\begin{align}
    \xkk &= \yk - \eta \gg(y_{k}, \xi_{k}) \label{eq:grad-step_1}\\
    y_k  &= \agk z_k + (1 - \agk)\xk  \label{eq:grad-step_2}\\
    \vkk &= \bgk z_k + (1 - \bgk) \yk - \gk \eta \gg(y_{k}, \xi_{k})\label{eq:grad-step_3}
\end{align}
where $\gg(y_{k}, \xi_{k})$ is a biased gradient oracle (see Definition \ref{Definition_1}) and updates of the parameters:
\begin{align}
    \gamma_{k} &= \frac{1}{\rho} \left( 1 + \frac{\zeta_k (1 - \alpha_k)}{\alpha_k} \right),\label{eq:gamma-update}\\
    \bgk &\geq 1 - \gk \mu \eta,\label{eq:beta-update}\\
    a_{k + 1} &= \gamma_k \sqrt{\eta \rho} b_{k + 1},\label{eq:a-update}\\
    b_{k + 1} &= \frac{b_k}{\sqrt{\zeta_k}},\label{eq:b-update}\\
    \alpha_k &= \frac{\gamma_{k} \bgk b_{k + 1}^2 \eta}{\gamma_{k} \bgk b_{k + 1}^2 \eta + a_k^2}\label{eq:alpha-update}.
\end{align} 
We now prove the following lemma assuming that the function $f(\cdot)$ is convex~and~smooth. 
\begin{lemma}\label{lemma:result}
Let convex function satisfy Assumption \ref{ass:L_smooth} and the gradient oracle $\gg(y_k, \xi_k)$ of Definition \ref{Definition_1} satisfies Assumptions \ref{Ass:Bounded_bias} and \ref{Ass:Gradient_noise}. Then, using the updates in Equation~\eqref{eq:grad-step_1}-\eqref{eq:grad-step_3} and setting the parameters according to Equations~\ref{eq:gamma-update}-~\ref{eq:b-update}, if $\eta \leq \frac{1}{\rho L}$ and $\Phi_{k}:= \expect{f(x_{k})} - f^*$, then the following relation holds:
\begin{align*}
b_N^2 \gamma_{N-1}^2 \Phi_{N} &\leq  \frac{a_{0}}{\rho \eta} \Phi_0 +  \frac{b_{0}^2}{\rho \eta} R_0^2  + \sum_{k=0}^{N-1}  \frac{b_{k+1}^2 \gamma_k}{\rho} \Tilde{R} \norms{\expect{\gg_k} - \nabla f(y_k)} \\
&+  \sum_{k=0}^{N-1} \frac{a_{k+1}^2 \sigma^2}{2 \rho^2} + \sum_{k=0}^{N-1} \frac{1}{2 \rho} \norms{\expect{\gg_k} - \nabla f(y_k)}^2.
\end{align*}
\end{lemma}
\begin{proof}
    Let $\gg_{k} = \gg(y_k, \xi_k)$ and $R_{k+1} = \norms{z_{k+1} - x^{*}}$, then from equation \eqref{eq:grad-step_3}:
    \begin{align*}
        R_{k+1}^2 &= \norms{\zeta_k z_k + (1-\zeta_k) y_k - x^{*} - \gamma_k \eta \gg_{k}}^2\\
         &= \norms{\zeta_k z_k + (1-\zeta_k) y_k - x^{*}}^2 + \gamma_k^2 \eta^2 \norms{\gg_{k}}^2 + 2 \gamma_k \eta \dotprod{x^{*} - \zeta_k z_k - (1 - \zeta_k)y_k}{\gg_k}.
    \end{align*}
    Taking expectation wrt to $\xi_k$:
    \begin{eqnarray}
        \expect{R_{k+1}^2} &=& \expect{\norms{\zeta_k z_k + (1-\zeta_k) y_k - x^{*}}^2} + \gamma_k^2 \eta^2 \expect{\norms{\gg_{k}}^2} 
        \nonumber\\
        && \quad + 2 \gamma_k \eta \expect{\dotprod{x^{*} - \zeta_k z_k - (1 - \zeta_k)y_k}{\gg_k}}
        \nonumber\\
        &\overset{\eqref{eq:Gradient_noise}}{\leq}& \norms{\zeta_k z_k + (1-\zeta_k) y_k - x^{*}}^2 + \gamma_k^2 \eta^2 \rho \norms{\nabla f(y_k)}^2 
        \nonumber\\
        && \quad + 2 \gamma_k \eta \dotprod{x^{*} - \zeta_k z_k - (1 - \zeta_k)y_k}{\expect{\gg_k}} + \gamma_k^2 \eta^2 \sigma^2
        \nonumber \\
        &=& \norms{\zeta_k ( z_k - x^{*}) + (1-\zeta_k) (y_k - x^{*})}^2 + \gamma_k^2 \eta^2 \rho \norms{\nabla f(y_k)}^2 
        \nonumber\\
        && \quad  + 2 \gamma_k \eta \dotprod{x^{*} - \zeta_k z_k - (1 - \zeta_k)y_k}{\expect{\gg_k}} + \gamma_k^2 \eta^2 \sigma^2
        \nonumber\\
        &\overset{\circledOne}{\leq}& \zeta_k \norms{ z_k - x^{*}}^2 + (1-\zeta_k) \norms{y_k - x^{*}}^2 + \gamma_k^2 \eta^2 \rho \norms{\nabla f(y_k)}^2
        \nonumber\\
        && \quad   + 2 \gamma_k \eta \dotprod{x^{*} - \zeta_k z_k - (1 - \zeta_k)y_k}{\expect{\gg_k}} + \gamma_k^2 \eta^2 \sigma^2
        \nonumber\\
        &=& \zeta_k R_k^2 + (1-\zeta_k) \norms{y_k - x^{*}}^2 + \gamma_k^2 \eta^2 \rho \norms{\nabla f(y_k)}^2  
        \nonumber\\
        && \quad  + 2 \gamma_k \eta \dotprod{x^{*} - \zeta_k z_k - (1 - \zeta_k)y_k}{\expect{\gg_k}} + \gamma_k^2 \eta^2 \sigma^2
        \nonumber\\
        &=& \zeta_k R_k^2 + (1-\zeta_k) \norms{y_k - x^{*}}^2 + \gamma_k^2 \eta^2 \rho \norms{\nabla f(y_k)}^2 
        \nonumber\\
        && \quad + 2 \gamma_k \eta \dotprod{\zeta_k (y_k - z_k) + x^{*} - y_k}{\expect{\gg_k}}  + \gamma_k^2 \eta^2 \sigma^2
        \nonumber\\
        &\overset{\eqref{eq:grad-step_2}}{=}& \zeta_k R_k^2 + (1-\zeta_k) \norms{y_k - x^{*}}^2 + \gamma_k^2 \eta^2 \rho \norms{\nabla f(y_k)}^2 + \gamma_k^2 \eta^2 \sigma^2
        \nonumber\\
        && \quad  + 2 \gamma_k \eta \dotprod{\frac{\zeta_k (1 - \alpha_k)}{\alpha_k} (x_k - y_k) + x^{*} - y_k}{\expect{\gg_k}} 
        \nonumber\\
        &=& \zeta_k R_k^2 + (1-\zeta_k) \norms{y_k - x^{*}}^2 + \gamma_k^2 \eta^2 \rho \norms{\nabla f(y_k)}^2 + \gamma_k^2 \eta^2 \sigma^2 
        \nonumber\\
        && \quad + 2 \gamma_k \eta \left( \frac{\zeta_k (1 - \alpha_k)}{\alpha_k} \dotprod{\expect{\gg_k}}{(x_k - y_k)} + \dotprod{\expect{\gg_k}}{x^{*} - y_k}\right)
        \nonumber\\
        &\overset{\circledTwo}{\leq}& \zeta_k R_k^2 + (1-\zeta_k) \norms{y_k - x^{*}}^2 + \gamma_k^2 \eta^2 \rho \norms{\nabla f(y_k)}^2 + \gamma_k^2 \eta^2 \sigma^2  
        \nonumber\\
        && \quad + 2 \gamma_k \eta \left( \frac{\zeta_k (1 - \alpha_k)}{\alpha_k} \left[ f(x_{k}) - f( y_{k})\right] + f(x^{*}) - f(y_k) \right)
        \nonumber\\
        && \quad + 2 \gamma_k \eta \left( \frac{\zeta_k (1 - \alpha_k)}{\alpha_k} \left[ \dotprod{\expect{\gg_k} - \nabla f(y_k)} {(x_k - y_k)}\right] \right)
        \nonumber\\
        && \quad + 2 \gamma_k \eta \dotprod{\expect{\gg_k} - \nabla f(y_k)} {(x^{*} - y_k)}, \label{eq:prom_itog}
    \end{eqnarray}
    where in inequality $\circledOne$ we used convexity of $\norms{\cdot}^2$ and in inequality $\circledTwo$ we used convexity~of~$f(\cdot)$: 
    \begin{eqnarray*}
        \dotprod{\expect{\gg_k}}{x_k - y_k} &=& \dotprod{\expect{\gg_k} - \nabla f(y_k)}{x_k - y_k} + \dotprod{\nabla f(y_k)}{x_k - y_k}
        \\
        &\leq& \dotprod{\expect{\gg_k} - \nabla f(y_k)}{x_k - y_k} + f(x_k) - f(y_k).
    \end{eqnarray*}
    By Lipschitz continuity of the gradient (see Assumption \ref{ass:L_smooth}):
    \begin{eqnarray*}
        f(x_{k+1}) - f(y_k) \leq \dotprod{\nabla f(y_k)}{x_{k+1} - y_{k}} + \frac{L}{2} \norms{x_{k+1} - y_{k}}^2 \leq - \eta \dotprod{\nabla f(y_k)}{\gg_k} + \frac{L \eta^2}{2} \norms{\gg_{k}}^2.
    \end{eqnarray*}
    Taking expectation wrt $\xi_k$ and choosing the parameter $\eta \leq \frac{1}{2 \rho L}$ we have:
    \begin{align*}
        \expect{f(x_{k+1} - f(y_k))} &\leq - \eta \dotprod{\nabla f(y_k)}{\expect{\gg_k}} + \frac{L \eta^2}{2} \expect{\norms{\gg_{k}}^2}
        \\
        &\! \overset{\eqref{eq:Gradient_noise}}{\leq} - \eta \dotprod{\nabla f(y_k)}{\expect{\gg_k}} + \frac{L \eta^2 \rho}{2} \norms{\nabla f(y_k)}^2 + \frac{L \eta^2 \sigma^2}{2}
        \\
        &= - \eta \norms{\nabla f(y_k)}^2 + \eta \dotprod{\nabla f(y_k)}{\nabla f(y_k) - \expect{\gg_k}} 
        \\
        & \quad + \frac{L \eta^2 \rho}{2} \norms{\nabla f(y_k)}^2 + \frac{L \eta^2 \sigma^2}{2}
        \\
        &\overset{\eqref{eq:fenchel_young_inequality}}{\leq} - \frac{\eta}{2} \norms{\nabla f(y_k)}^2 + \frac{\eta}{2} \norms{\expect{\gg_k} - \nabla f(y_k)}^2 
        \\
        & \quad + \frac{L \eta^2 \rho}{2} \norms{\nabla f(y_k)}^2 + \frac{L \eta^2 \sigma^2}{2}
        \\
        &= \left( - \frac{\eta}{2} + \frac{L \eta^2 \rho}{2} \right) \norms{\nabla f(y_k)}^2 + \frac{\eta}{2} \norms{\expect{\gg_k} - \nabla f(y_k)}^2 + \frac{L \eta^2 \sigma^2}{2}
        \\
        &\leq  - \frac{\eta}{4} \norms{\nabla f(y_k)}^2 + \frac{\eta}{2} \norms{\expect{\gg_k} - \nabla f(y_k)}^2 + \frac{L \eta^2 \sigma^2}{2}.
    \end{align*}
    Consequently, we can obtain upper bound:
    \begin{equation}\label{eq:upper_normgrad}
        \norms{\nabla f(y_k)}^2 \leq \frac{4}{\eta} \expect{f(y_k) - f(x_{k+1})} + 2 \norms{\expect{\gg_k} - \nabla f(y_k)}^2 + 2 L \eta \sigma^2.
    \end{equation}
    Substituting the upper bound gradient norm \eqref{eq:upper_normgrad} in the initial inequality \eqref{eq:prom_itog} we have
    \begin{align*}
        \expect{R_{k+1}^2} &\leq  \zeta_k R_k^2 + (1-\zeta_k) \norms{y_k - x^{*}}^2 + 4 \gamma_k^2 \eta \rho \expect{f(y_k) - f(x_{k+1})} 
        \\
        & \quad + 2 \gamma_k \eta \left( \frac{\zeta_k (1 - \alpha_k)}{\alpha_k} \left[ f(x_{k}) - f( y_{k})\right] + f(x^{*}) - f(y_k) \right) 
        \\
        & \quad + 2 \gamma_k \eta  \frac{\zeta_k (1 - \alpha_k)}{\alpha_k}  \dotprod{\expect{\gg_k} - \nabla f(y_k)} {(x_k - y_k)}
        \\
        & \quad + \gamma_k^2 \eta^2 \sigma^2 + 2 L \gamma_k^2 \eta^3 \rho \sigma^2 + 2 \gamma_k^2 \eta^2 \rho \norms{\expect{\gg_k} - \nabla f(y_k)}^2
        \\
        & \quad + 2 \gamma_k \eta \dotprod{\expect{\gg_k} - \nabla f(y_k)} {(x^{*} - y_k)}
        \\
        &\overset{\circledOne}{\leq} \zeta_k R_k^2 + (1-\zeta_k) \norms{y_k - x^{*}}^2 + 4 \gamma_k^2 \eta \rho \expect{f(y_k) - f(x_{k+1})} 
        \\
        & \quad + 2 \gamma_k \eta \left( \frac{\zeta_k (1 - \alpha_k)}{\alpha_k} \left[ f(x_{k}) - f( y_{k})\right] + f(x^{*}) - f(y_k) \right) 
        \\
        & \quad + 2 \gamma_k \eta \frac{\zeta_k (1 - \alpha_k)}{\alpha_k}  \dotprod{\expect{\gg_k} - \nabla f(y_k)} {(x_k - y_k)}
        \\
        & \quad + 2 \gamma_k^2 \eta^2 \sigma^2  + 2 \gamma_k^2 \eta^2 \rho \norms{\expect{\gg_k} - \nabla f(y_k)}^2
        \\
        & \quad + 2 \gamma_k \eta  \dotprod{\expect{\gg_k} - \nabla f(y_k)} {(x^{*} - y_k)}
        \\
        &= \zeta_k R_k^2 + (1-\zeta_k) \norms{y_k - x^{*}}^2  + 2 \gamma_k^2 \eta^2 \sigma^2
        \\
        & \quad + f(y_k) \left( 4 \gamma_k^2 \eta \rho - 2 \gamma_k \eta  \frac{\zeta_k (1 - \alpha_k)}{\alpha_k} - 2 \gamma_k \eta \right)
        \\
        & \quad - 4 \gamma_k^2 \eta \rho \expect{f(x_{k+1})} + 2 \gamma_k \eta f(x^{*}) + \left(2 \gamma_k \eta  \frac{\zeta_k (1 - \alpha_k)}{\alpha_k} \right) f(x_{k}) 
        \\
        & \quad + 2 \gamma_k^2 \eta^2 \rho \norms{\expect{\gg_k} - \nabla f(y_k)}^2 + 2 \gamma_k \eta \dotprod{\expect{\gg_k} - \nabla f(y_k)} {(x^{*} - y_k)}
        \\
        & \quad + 2 \gamma_k \eta  \frac{\zeta_k (1 - \alpha_k)}{\alpha_k}  \dotprod{\expect{\gg_k} - \nabla f(y_k)} {(x_k - y_k)}
        \\
        &= \zeta_k R_k^2 + (1-\zeta_k) \norms{y_k - x^{*}}^2 + 2 \gamma_k^2 \eta^2 \sigma^2
        \\
        & \quad + f(y_k) \left( 4 \gamma_k^2 \eta \rho - 2 \gamma_k \eta  \frac{\zeta_k (1 - \alpha_k)}{\alpha_k} - 2 \gamma_k \eta \right)
        \\
        & \quad - 4 \gamma_k^2 \eta \rho \expect{f(x_{k+1})} + 2 \gamma_k \eta f(x^{*}) + \left(2 \gamma_k \eta  \frac{\zeta_k (1 - \alpha_k)}{\alpha_k} \right) f(x_{k}) 
        \\
        & \quad + 2 \gamma_k^2 \eta^2 \rho \norms{\expect{\gg_k} - \nabla f(y_k)}^2
        \\
        & \quad + 2 \gamma_k \eta \dotprod{ \frac{\zeta_k (1 - \alpha_k)}{\alpha_k} (x_k - y_k) + x^{*} - y_k}{\expect{\gg_k} - \nabla f(y_k)},
    \end{align*}
    where in inequality $\circledOne$ we used the fact that $\eta \leq \frac{1}{2 \rho L}$.

    Since $\zeta_k \geq 1$ and $\gamma_k = \frac{1}{2 \rho} \left( 1 + \frac{\zeta_k (1 - \alpha_k)}{\alpha_k} \right)$, we have
    \begin{align*}
        \expect{R_{k+1}^2} &\leq \zeta_k R_k^2 +  2 \gamma_k^2 \eta^2 \sigma^2 - 4 \gamma_k^2 \eta \rho \expect{f(x_{k+1})} + 2 \gamma_k \eta f(x^{*})  
        \\
        & \quad + 2 \gamma_k \eta \dotprod{ \frac{\zeta_k (1 - \alpha_k)}{\alpha_k} (x_k - y_k) + x^{*} - y_k}{\expect{\gg_k} - \nabla f(y_k)}
        \\
        & \quad + \left(2 \gamma_k \eta  \frac{\zeta_k (1 - \alpha_k)}{\alpha_k} \right) f(x_{k}) + 2 \gamma_k^2 \eta^2 \rho \norms{\expect{\gg_k} - \nabla f(y_k)}^2
        \\
        &\!\! \overset{\eqref{eq:grad-step_2}}{=} \zeta_k R_k^2 +  2 \gamma_k^2 \eta^2 \sigma^2 - 4 \gamma_k^2 \eta \rho \expect{f(x_{k+1})} + 2 \gamma_k \eta f(x^{*}) 
        \\
        & \quad + 2 \gamma_k \eta \dotprod{ \zeta_k (y_k - z_k) + x^{*} - y_k}{\expect{\gg_k} - \nabla f(y_k)}
        \\
        & \quad + \left(2 \gamma_k \eta  \frac{\zeta_k (1 - \alpha_k)}{\alpha_k} \right) f(x_{k}) + 2 \gamma_k^2 \eta^2 \rho \norms{\expect{\gg_k} - \nabla f(y_k)}^2
        \\
        &= \zeta_k R_k^2 +  2 \gamma_k^2 \eta^2 \sigma^2  - 4 \gamma_k^2 \eta \rho \expect{f(x_{k+1})} + 2 \gamma_k \eta f(x^{*})  
        \\
        & \quad + 2 \gamma_k \eta \dotprod{ (\zeta_k - 1) (y_k - z_k) + x^{*} - z_k}{\expect{\gg_k} - \nabla f(y_k)}
        \\
        & \quad + \left(2 \gamma_k \eta  \frac{\zeta_k (1 - \alpha_k)}{\alpha_k} \right) f(x_{k}) + 2 \gamma_k^2 \eta^2 \rho \norms{\expect{\gg_k} - \nabla f(y_k)}^2
        \\
        &\leq \zeta_k R_k^2 +  2 \gamma_k^2 \eta^2 \sigma^2 - 4 \gamma_k^2 \eta \rho \expect{f(x_{k+1})} + 2 \gamma_k \eta f(x^{*})  
        \\
        & \quad + 2 \gamma_k \eta \norms{ (\zeta_k - 1) (y_k - z_k) + x^{*} - z_k} \norms{\expect{\gg_k} - \nabla f(y_k)}
        \\
        & \quad + \left(2 \gamma_k \eta  \frac{\zeta_k (1 - \alpha_k)}{\alpha_k} \right) f(x_{k}) + 2 \gamma_k^2 \eta^2 \rho \norms{\expect{\gg_k} - \nabla f(y_k)}^2
        \\
        &\leq \zeta_k R_k^2 +  2 \gamma_k^2 \eta^2 \sigma^2 - 4 \gamma_k^2 \eta \rho \expect{f(x_{k+1})} + 2 \gamma_k \eta f(x^{*})  
        \\
        & \quad + 2 \gamma_k \eta \left((\zeta_k - 1)\norms{  (z_k - y_k )} + \norms{z_k- x^{*}}\right) \norms{\expect{\gg_k} - \nabla f(y_k)}
        \\
        & \quad + \left(2 \gamma_k \eta  \frac{\zeta_k (1 - \alpha_k)}{\alpha_k} \right) f(x_{k}) + 2 \gamma_k^2 \eta^2 \rho \norms{\expect{\gg_k} - \nabla f(y_k)}^2
        \\
        &\leq \zeta_k R_k^2 +  2 \gamma_k^2 \eta^2 \sigma^2 - 4 \gamma_k^2 \eta \rho \expect{f(x_{k+1})} 
        \\
        & \quad + 2 \gamma_k \eta f(x^{*}) + \left(2 \gamma_k \eta  \frac{\zeta_k (1 - \alpha_k)}{\alpha_k} \right) f(x_{k})
        \\
        & \quad  + 2 \gamma_k \eta \Tilde{R} \norms{\expect{\gg_k} - \nabla f(y_k)} + 2 \gamma_k^2 \eta^2 \rho \norms{\expect{\gg_k} - \nabla f(y_k)}^2,
    \end{align*}
    where $\Tilde{R} = \max_{k}\left\{ \norms{z_k - x^{*}}, \norms{z_k - y_k} \right\}$.

    Multiplying by $b_{k+1}^2$:
    \begin{align*}
        b_{k+1}^2 \expect{R_{k+1}^2} &\leq b_{k+1}^2 \zeta_k R_k^2 +  2 b_{k+1}^2 \gamma_k^2 \eta^2 \sigma^2 - 4 b_{k+1}^2 \gamma_k^2 \eta \rho \expect{f(x_{k+1})} 
        \\
        & \quad + 2 b_{k+1}^2 \gamma_k \eta  f(x^{*}) + \left(2 b_{k+1}^2 \gamma_k \eta  \frac{\zeta_k (1 - \alpha_k)}{\alpha_k} \right) f(x_{k})
        \\
        & \quad  + 2 b_{k+1}^2 \gamma_k \eta \Tilde{R} \norms{\expect{\gg_k} - \nabla f(y_k)} + 2 b_{k+1}^2 \gamma_k^2 \eta^2 \rho \norms{\expect{\gg_k} - \nabla f(y_k)}^2.
    \end{align*}
    Since 
    \begin{equation*}
        b^2_{k+1} \zeta_k \leq b^2_k; \;\;\; b^2_{k+1} \gamma^{2}_{k} \eta \rho = a_{k+1}^2; \;\;\; \frac{b_{k+1}^2 \gamma_k \eta \zeta_k (1 - \alpha_k)}{\alpha_k} = a_k^2
    \end{equation*} 
    and $b^{2}_{k+1} \gamma_k \eta - a^2_{k+1} + a^2_k = 0$ we have:
    \begin{align*}
        b_{k+1}^2 \expect{R_{k+1}^2} &\leq b_{k}^2 R_k^2 +  \frac{a_{k+1}^2 \eta \sigma^2}{\rho}  - 2 a_{k+1}^2 \expect{f(x_{k+1})} + 2 b_{k+1}^2 \gamma_k \eta f(x^{*}) + 2 a_{k}^2  f(x_{k})
        \\
        & \quad  + 2 b_{k+1}^2 \gamma_k \eta \Tilde{R} \norms{\expect{\gg_k} - \nabla f(y_k)} +  a_{k+1}^2 \eta \norms{\expect{\gg_k} - \nabla f(y_k)}^2
        \\
        &= b_{k}^2 R_k^2 +  \frac{a_{k+1}^2 \eta \sigma^2}{\rho}  - 2 a_{k+1}^2 \left[ \expect{f(x_{k+1})} - f^* \right]  + 2 a_{k}^2  \left[ f(x_{k}) - f^* \right] 
        \\
        & \quad + 2 \left[b_{k+1}^2 \gamma_k \eta - a_{k+1}^2 + a_{k}^2 \right] f(x^{*})
        \\
        & \quad  + 2 b_{k+1}^2 \gamma_k \eta \Tilde{R} \norms{\expect{\gg_k} - \nabla f(y_k)} +  a_{k+1}^2 \eta \norms{\expect{\gg_k} - \nabla f(y_k)}^2
        \\
        &= b_{k}^2 R_k^2 +  \frac{a_{k+1}^2 \eta \sigma^2}{\rho}  - 2 a_{k+1}^2 \left[ \expect{f(x_{k+1})} - f^* \right]  + 2 a_{k}^2  \left[ f(x_{k}) - f^* \right] 
        \\
        & \quad  + 2 b_{k+1}^2 \gamma_k \eta \Tilde{R} \norms{\expect{\gg_k} - \nabla f(y_k)} + a_{k+1}^2 \eta \norms{\expect{\gg_k} - \nabla f(y_k)}^2.
    \end{align*}
    By rearranging the terms and denoting $\Phi_{k}:= \expect{f(x_{k})} - f^*$ we get:
    \begin{align*}
        2 a_{k+1}^2 \Phi_{k+1} - 2 a_{k}^2 \Phi_{k} &\leq b_{k}^2 R_k^2 - b_{k+1}^2 \expect{R_{k+1}^2} + \frac{a_{k+1}^2 \eta \sigma^2}{\rho} 
        \\
        & \quad + 2 b_{k+1}^2 \gamma_k \eta \Tilde{R} \norms{\expect{\gg_k} - \nabla f(y_k)} +  a_{k+1}^2 \eta \norms{\expect{\gg_k} - \nabla f(y_k)}^2.
    \end{align*}
    By summing over $k$ we obtain: 
    \begin{align*}
        \sum_{k=0}^{N-1} 2 \left( a_{k+1} \Phi_{k+1} - a_{k} \Phi_k \right) &\leq \sum_{k=0}^{N-1} b_{k}^2 R_k^2 - b_{k+1}^2 \expect{R_{k+1}^2} +  \sum_{k=0}^{N-1} \frac{a_{k+1}^2 \eta \sigma^2}{\rho}
        \\
        & \quad  + \sum_{k=0}^{N-1} 2 b_{k+1}^2 \gamma_k \eta \Tilde{R} \norms{\expect{\gg_k} - \nabla f(y_k)} 
        \\
        & \quad + \sum_{k=0}^{N-1} a_{k+1}^2 \eta \norms{\expect{\gg_k} - \nabla f(y_k)}^2.
    \end{align*}
    Let's substitute $a_{k+1}^2$:
    \begin{align*}
         2  b_N^2 \gamma_{N-1}^2 \rho \eta \Phi_{N} &\leq  2 a_{0} \Phi_0 +  b_{0}^2 R_0^2 +  \sum_{k=0}^{N-1} \frac{a_{k+1}^2 \eta \sigma^2}{\rho} + \sum_{k=0}^{N-1} 2 b_{k+1}^2 \gamma_k \eta \Tilde{R} \norms{\expect{\gg_k} - \nabla f(y_k)} 
         \\
         & \quad + \sum_{k=0}^{N-1} a_{k+1}^2 \eta \norms{\expect{\gg_k} - \nabla f(y_k)}^2.
    \end{align*}
    Now dividing by $2 \rho \eta$ we have:
    \begin{align*}
         b_N^2 \gamma_{N-1}^2 \Phi_{N} &\leq  \frac{a_{0}}{\rho \eta} \Phi_0 +  \frac{b_{0}^2}{\rho \eta} R_0^2 +  \sum_{k=0}^{N-1} \frac{a_{k+1}^2 \sigma^2}{2 \rho^2} + \sum_{k=0}^{N-1}  \frac{b_{k+1}^2 \gamma_k}{\rho} \Tilde{R} \norms{\expect{\gg_k} - \nabla f(y_k)} 
         \\
         & \quad + \sum_{k=0}^{N-1} \frac{a_{k+1}^2}{2 \rho} \norms{\expect{\gg_k} - \nabla f(y_k)}^2.
    \end{align*}
\qed
\end{proof}

\begin{lemma}
Under the parameter setting according to Equations~\ref{eq:gamma-update}--\ref{eq:b-update}, the following relation is true:
\begin{align*}
\gk^2 - \gk \frac{1}{\rho}  = \gamma^2_{k-1}
\end{align*}
\begin{proof}
\begin{align}
\gk & = \frac{1}{2\rho} \left[ 1 + \frac{\bgk (1 - \agk)}{\agk} \right]  \nonumber \\
\gk^2 - \frac{\gk}{2\rho} & = \frac{\gk \bgk (1 - \agk)}{2\rho \agk} \nonumber \\
& = \frac{1}{2\eta \rho} \frac{\ak^2}{\bkk^2}   \nonumber \\
& = \frac{\bgk}{2\eta \rho} \frac{\ak^2}{\bk^2}  \nonumber \\
& = \frac{1}{2\eta \rho} \frac{\ak^2}{\bk^2} \nonumber \\
& = \frac{1 }{\eta \rho} \left( \gamma_{k-1} \sqrt{\eta \rho} \right)^2   \nonumber \\
& =  \gamma^2_{k-1}. \nonumber  
\end{align}
\qed
\end{proof}
\label{lemma:gamma}
\end{lemma}
Let us write the result of Lemma \ref{lemma:gamma} as:
\begin{equation*}
    \gamma^{2}_k - \frac{\gamma_k}{\rho} - \gamma_{k-1}^2 = 0.
\end{equation*}
Next we can find the parameter $\gamma_{k}$:
\begin{equation*}
    \gamma_{k} = \frac{\frac{1}{\rho} + \sqrt{\frac{1}{\rho^2} + 4 \gamma_{k-1}^2}}{2}.
\end{equation*}
Let $\gamma_{0} = 0$, then for all $k$ we have:
\begin{eqnarray*}
    \zeta_{k} &=& 1 \\
    b_{k+1} &=& b_{k} = b_0 = 1\\
    a_{k+1} &=& \gamma_{k} \sqrt{\eta \rho} b_0 \Rightarrow  a_{k+1} = \gamma_{k} \sqrt{\eta \rho}.
\end{eqnarray*}
The above equation implies that $a_0 = 0$. Then using the result of Lemma \ref{lemma:result} and by induction $\gamma_{k} \geq \frac{k}{2 \rho}$ we have:
\begin{align*}
         \frac{N^2}{4 \rho^2} \Phi_{N} &\leq   \frac{a_{0}}{\rho \eta} \Phi_0 +  \frac{b_{0}^2}{\rho \eta} R_0^2 +  \sum_{k=0}^{N-1} \frac{a_{k+1}^2 \sigma^2}{2 \rho^2} + \sum_{k=0}^{N-1}  \frac{b_{k+1}^2 \gamma_k}{\rho} \Tilde{R} \norms{\expect{\gg_k} - \nabla f(y_k)} 
         \\
         & \quad + \sum_{k=0}^{N-1} \frac{a_{k+1}^2}{2 \rho} \norms{\expect{\gg_k} - \nabla f(y_k)}^2
         \\
         &\leq   \frac{1}{\rho \eta} R_0^2 +  \sum_{k=0}^{N-1} \frac{\gamma_{k}^2 \eta \sigma^2}{2 \rho} + \sum_{k=0}^{N-1}  \frac{ \gamma_k}{\rho} \Tilde{R} \norms{\expect{\gg_k} - \nabla f(y_k)} 
         \\
         & \quad + \sum_{k=0}^{N-1} \frac{\gamma_{k}^2 \eta}{2} \norms{\expect{\gg_k} - \nabla f(y_k)}^2
         \\
         &\leq   \frac{1}{\rho \eta} R_0^2 +  \sum_{k=0}^{N-1} \frac{k^2 \eta \sigma^2}{8 \rho^3} + \sum_{k=0}^{N-1}  \frac{ k}{2 \rho^2} \Tilde{R} \norms{\expect{\gg_k} - \nabla f(y_k)} 
         \\
         & \quad + \sum_{k=0}^{N-1} \frac{k^2}{4 L\rho^2} \norms{\expect{\gg_k} - \nabla f(y_k)}^2
         \\
         &\leq   \frac{1}{\rho \eta} R_0^2 + \frac{k^3 \eta \sigma^2}{24 \rho^3} + \frac{ N^2}{4 \rho^2} \Tilde{R} \delta +  \frac{N^3}{12 L \rho^2} \delta^2.
    \end{align*}

    Now we can get the convergence rate
    \begin{eqnarray*}
        \Phi_{N} &\leq& \frac{4 \rho}{N^2 \eta} R_0^2 + \frac{N \eta \sigma^2}{6 \rho} +  \Tilde{R} \delta +  \frac{N}{3L} \delta^2.
    \end{eqnarray*}

    By adding batching, given that $\rho_{B} = \max \{ 1, \frac{\rho}{B} \}$, $\sigma^2_B = \frac{\sigma^2}{B}$ and $R = R_0$ we have the convergence rate for accelerated SGD with biased gradient oracle and parameter $\eta \lesssim \frac{1}{\rho_B L} $:
    \begin{equation*}
        \boxed{\expect{f(x_{N})} - f^* \lesssim \frac{\rho_B^2 L R^2}{N^2} + \frac{N \sigma_B^2}{\rho_B^2 L} +  \Tilde{R} \delta +  \frac{N}{L} \delta^2.}
    \end{equation*}










\end{document}